\documentclass[12pt,a4paper,leqno,oneside]{article}
\usepackage{amsmath,amssymb,amsfonts,amsthm}
\usepackage[active]{srcltx}
\usepackage[dvips]{color}
\newcommand{\Pra}[1]{\Pr\left\{#1\right\}}
\newcommand{\PP}{\mathbb{P}}
\newcommand{\kr}[2]{R^{(#1)}_{#2}}
\newcommand{\kra}[2]{X^{(#1)}_{#2}}
\newcommand{\kraw}[2]{Y^{(#1)}_{#2}}
\newcommand{\dtv}[2]{d_{TV}\left(#1,#2\right)}
\newcommand{\whp}{with high probability}
\newcommand{\Gnm}[2]{\mathcal{G}_{#1}\left(n, m, #2\right)}
\newcommand{\Hnm}[2]{\mathcal{H}_{#1}\left(n, m, #2\right)}
\newcommand{\Gn}[1]{G\left(n,#1\right)}
\newcommand{\Hn}[2]{H_{#1}\left(n,#2\right)}
\newcommand{\Hk}[2]{H^{(#1)}\left(n,#2\right)}

\newcommand{\Hnq}[2]{H^{(#1)}_*\left(C'nq,#2\right)}

\newcommand{\Gkn}[2]{G^{(#1)}\left(n,#2\right)}
\newcommand{\czesc}[1]
{\  \medskip\\
{\noindent \bf #1}
\medskip\\}

\newcommand{\M}{\mathcal{M}}

\newcommand{\D}[2]{\mathbb{D}^{(#1)}_{#2}}

\newcommand{\V}{\mathcal{V}}
\newcommand{\X}[1]{\mathcal{X}_{#1}}

\newcommand{\W}{\mathcal{W}}
\newcommand{\Po}{\textrm{Po}}
\newcommand{\Bin}{\textrm{Bin}}
\newcommand{\E}{\mathbb{E}}
\newcommand{\eps}{\varepsilon}
\newcommand{\coup}[1]{\preccurlyeq_{1-#1}}
\newcommand{\coupling}[1]{\preccurlyeq_{#1}}

\newtheorem{tw}{Theorem}

\newtheorem{lem}{Lemma}

\newtheorem{fact}{Fact}

\theoremstyle{plain}

\begin{document}

\author{Katarzyna Rybarczyk$^*$}
\title{Equivalence of a random intersection graph and $G(n,p)$}
\date{}

\maketitle
\begin{center}
$^*$Faculty of Mathematics and Computer
Science, Adam Mickiewicz University,\\ 60--769 Pozna\'n, Poland
\end{center}

\begin{abstract}
We solve the conjecture of Fill, Scheinerman and Singer-Cohen posed in \cite{GpEquivalence} and show equivalence of sharp threshold functions of a random intersection graph $\Gnm{}{p}$	with $m \ge n^3$  and a graph $\Gn{\hat{p}}$ with independent edges. Moreover we prove sharper equivalence results under some additional assumptions. 
\end{abstract}

{\bfseries keywords:} random intersection graph, equivalence, graph properties

\section{Introduction}\label{Introduction}

In a random intersection graph there is a set of vertices $\V$ and an auxiliary
set of features $\W$. Each vertex $v\in\V$ is assigned a subset of features
$W(v)\subseteq\W$ according to a given probability measure. Two vertices $v_1$, $v_2$ are adjacent in a random intersection graph if and only if
$W(v_1)\cap W(v_2)\neq\emptyset$. A general model of a random intersection graph, in which each vertex is assigned a subset of features
$W(v)\subseteq\W$ chosen uniformly at random from all $d$--element subsets, where
the cardinality $d$ is determined according to the arbitrarily given
probability distribution, was introduced in \cite{RIGGodehardt1}.

We concentrate on analysing properties of a random intersection graph in which the cardinality $d$ is chosen according to the binomial distribution. Namely, we investigate properties of a random intersection graph $\Gnm{}{p}$ introduced in \cite{GpSubgraph,SingerPhD}. $\Gnm{}{p}$ is
 a graph with number of vertices $|\V|=n$,
number of features $|\W|=m$, in which each feature $w$ is added to $W(v)$ with probability $p$ independently
for all $v\in\V$ and $w\in\W$ (i.e. $\Pra{w\in W(v)}=p$). However, to some extent, the results obtained may be generalised to other random intersection graph models due to equivalence theorems proved in Section~4 in \cite{WSNphase2}.

The general model of a random intersection graph has attracted lately much attention, mainly due to its wide applications such as: ''gate matrix layout'' for VLSI design (see e.g. \cite{GpSubgraph}), cluster analysis and classification (see e.g. \cite{RIGGodehardt1}), analysis of complex networks (see e.g. \cite{RIGTunableDegree, RIGDegree2}), secure wireless networks (see e.g. \cite{WSNWlosi,WSNphase2}) and epidemics (\cite{GpEpidemics}). On the wave of interest many articles concerning $\Gnm{}{p}$ have appeared. Therefore an important issue is to indicate for which parameters $\Gnm{}{p}$ differs significantly from well known random graph models and as a consequence is worth studying. 
The first article on the topic was written by Scheinerman, Fill and Singer--Cohen \cite{GpEquivalence}. In \cite{GpEquivalence} the authors described differences and similarities between $\Gnm{}{p}$ and random graph $G(n,\hat{p})$ in which each edge appears independently with probability $\hat{p}$ ($\hat{p}$ was set to be approximately $\Pra{(v_1,v_2)\in E(\Gnm{}{p})}$). The main aim of this article is to extend results obtained by Scheinerman, Fill and Singer--Cohen and to solve their conjecture.

The main theorem in \cite{GpEquivalence} states that for $m=\lfloor n^{\alpha}\rfloor$ and $\alpha>6$ graphs $G(n,\hat{p})$ and $\Gnm{}{p}$ have asymptotically the same properties. Moreover, it is pointed out that the theorem may be extended to smaller values of $\alpha$ if additional assumptions about $p$ are made. The proof is based on the fact that for large $\alpha$ and relevant values of $p$, with probability tending to one as $n\to \infty$, there are no features assigned to more than two vertices and therefore the dependency between edges is asymptotically negligible. The authors of \cite{GpEquivalence} suggest that the equivalence theorem is true for all properties for $3\le \alpha\le 6$, i.e. in the case where the number of vertices assigned to each feature is still small.

The above mentioned result and conjecture are consistent with a simple observation that the number of vertices to which a given feature $w$ is assigned has essential impact on dependency between edge appearance in $\Gnm{}{p}$. An edge set of a random intersection graph $\Gnm{}{p}$ is a union of cliques with vertex sets $V(w):=\{v\in\V:w\in W(v)\}$, $w\in\W$. Therefore we may divide the set of edges of $\Gnm{}{p}$ according to the size of the clique in which the edges are contained.
Let $k\ge 2$. We denote by $\Gnm{k}{p}$ a graph with vertex set $\V$ and edge set $\{(v_1,v_2):\exists_{w} v_1,v_2\in V(w)\text{ and }|V(w)|=k\}$. Alternatively we may define $\Gnm{k}{p}=G(\Hnm{k}{p})$, where $\Hnm{k}{p}$ is a hypergraph with vertex set $\V$ and edge set $\{(v_1,v_2,\ldots,v_k): \exists_{w} V(w)=\{v_1,v_2,\ldots,v_k\}\}$ and for a hypergraph $\mathcal{H}$ a graph $G\mathcal{H}$ is a graph with the same vertex set as $\mathcal{H}$ and edge set consisting of those pairs of vertices which are contained in at least one edge of $\mathcal{H}$. Under this notation $E(\Gnm{}{p})=\bigcup_{k=2}^{m}E(G\Hnm{k}{p})=\bigcup_{k=2}^{m}E(\Gnm{k}{p})$. 
In \cite{GpEquivalence} it is shown that for some $m$ and $p$ graphs $\Gnm{}{p}$, $\Gnm{2}{p}$, $G(n,\hat{p})$ are  asymptotically almost the same. To be precise $\Gnm{k}{p}$ are empty for $k\ge 3$ with probability tending to one as $n\to \infty$ (we say {\whp}) and the edges in $\Gnm{2}{p}$ are almost independent. 

The authors in \cite{GpEquivalence} support the conjecture for $3\le \alpha\le 6$ by results concerning threshold functions for some properties of $\Gnm{}{p}$. However, it should be pointed out that if there exists $C>0$ such that 
\begin{equation}\label{p2}
p\ge C \left(1/n\sqrt[3]{m}\right), 
\end{equation}
then the expected number of edges in $\Gnm{3}{p}$ tends to a constant or even to infinity. Therefore one may expect that the structure of $\Gnm{}{p}$ and $\Gn{\hat{p}}$ differs. Namely, though the number of triangles in $\Gnm{2}{p}$ may make dominating contribution, the impact of triangles contained in $\Gnm{3}{p}$ on the structure of a random intersection graph cannot be omitted. As an example we may state the fact that for $\alpha=3$ the number of triangles in $\Gnm{}{p}$ and $G(n,\hat{p})$ on the threshold of appearance (i.e. for $p=c/n^{2}$ and $\hat{p}\sim mp^2= c^2/n$) has the Poisson distribution with parameters $(c^3+c^6)/3!$ and $c^6/3!$, respectively (see \cite{GpPoissonCliques}). For larger values of $\alpha$ the expected number of triangles in $\Gnm{}{p}$ and $G(n,\hat{p})$ may also differ significantly. The same is true for cliques of size four contained in $\Gnm{4}{p}$. In fact $\Gnm{k}{p}$ should be rather compared with $G\Hn{k}{\hat{p_k}}$, where $\hat{p_k}$ is approximately the probability that for given $\{v_1,\ldots,v_k\}\subseteq \V$ there exists $w$ such that $V(w)=\{v_1,\ldots,v_k\}$, $\Hn{k}{\hat{p_k}}$ is a $k$-uniform random hypergraph with each edge appearing independently with probability $\hat{p_k}$ and $G\Hn{k}{\hat{p_k}}$ is defined as above.
The above observation leads us to the conclusion that the equivalence theorem may not be stated for $3\le \alpha\le 6$ in such a general form as it was for $\alpha > 6$. Therefore we draw our attention to the case of monotone properties. The concept of restriction of the equivalence theorems to the class of monotone properties has already been developed while examining the equivalence of $G(n,\hat{p})$ and $G(n,M)$ (see \cite{KsiazkaBollobas,KsiazkaJLR,EREquivalence}). 

The article is organised as follows. In Section~\ref{SectionResult} we state and discuss the results. Basic definitions, auxiliary facts and lemmas are given in Section~\ref{SectionAuxiliary}. Section~\ref{SectionCoupling} includes the proof of a lemma which relates $\Gnm{}{p}$ to $\Gn{\hat{p}}$. The proofs of the main theorems are given in Section~\ref{SectionOutline}. For completeness, the last section called Appendix is added. It includes long proofs which have been omitted for clarity of considerations.

Throughout the article all limits are taken as $n\to \infty$. We also use standard Landaus notation $O(\cdot),\Theta(\cdot),\Omega(\cdot),o(\cdot),\sim$ (see for example \cite{KsiazkaJLR}) and we use the phrase '{\whp}' to  say with probability tending to one as $n\to \infty$.

\section{Result}\label{SectionResult}

In our considerations we draw our attention to $\Gnm{}{p}$ for
\begin{equation}\label{RownanieGlowne}
\Omega\left(\frac{1}{n\sqrt[3]{m}}\right) = p = O\left(\sqrt{\frac{\ln n}{m}}\right).
\end{equation}
For values of $p$ significantly larger than $\sqrt{\frac{\ln n}{m}}$ a graph $\Gnm{}{p}$ is {\whp} the complete graph on $n$ vertices (see \cite{GpEquivalence,SingerPhD}). Moreover if  
\begin{equation*}
p=o\left(\frac{1}{n\sqrt[3]{m}}\right)
\end{equation*}
then {\whp} $\Gnm{k}{p}$ are empty for all $k\ge 3$. Therefore a slight modification of the proof from \cite{GpEquivalence} implies that $\Gnm{}{p}$ and $G(n,\hat{p})$ are asymptotically equivalent for all graph properties. 
In fact the following equivalence theorem may be stated. 
\begin{tw}\label{alpha6}
 Let $a\in [0;1]$, $\mathcal{A}$ be any graph property, $p=o\left(\frac{1}{n\sqrt[3]{m}}\right)$ and 
$$
\hat{p}=1-\exp\left(-mp^{2}(1-p)^{n-2}\right).
$$ 
Then
$$
\Pra{\Gn{\hat{p}}\in\mathcal{A}}\to a
$$
if and only if
$$
\Pra{\Gnm{}{p}\in \mathcal{A}}\to a.
$$
\end{tw}

The main result of the article implies equivalence of the models for monotone properties. Most important properties such as connectivity, having the largest component of size at least $k$, containment of a perfect matching or containment of a  given graph as a subgraph are included in the wide family of monotone properties. Let $\mathcal{G}$ be a family of graphs with vertex set $\V$. We call $\mathcal{A}\subseteq \mathcal{G}$ an increasing (decreasing) property if $\mathcal{A}$ is closed under isomorphism and $G\in \mathcal{A}$ implies $G'\in \mathcal{A}$ for all $G'$ such that $E(G)\subseteq E(G')$ ($E(G')\subseteq E(G)$).

\begin{tw}\label{Twierdzenie3}
Let  $a\in [0;1]$,  $m=n^{\alpha}$ for $\alpha\ge 3$ and $\mathcal{A}$ be any monotone property. 
\begin{itemize} 
\item[(i)] Let  $p$ be as in \eqref{RownanieGlowne}
 and $1/n\sqrt[3]{m}=o(p)$  for $\alpha=3$.\\
If
$$
\Pra{\Gn{1-\exp(-mp^2(1-p)^{n-2})}\in\mathcal{A}}\to a
$$
and for all $\eps=\eps(n)\to 0$
$$
\Pra{\Gn{(1+\eps)(1-\exp(-mp^2(1-p)^{n-2}))}\in\mathcal{A}}\to a, 
$$
then 
$$
\Pra{\Gnm{}{p}\in\mathcal{A}}\to a.
$$
\item[(ii)]
Let $\hat{p}=\hat{p}(n)=\Omega(n^{-2}m^{1/3})$ for $\alpha>3$, $n^{-2}m^{1/3}=o(\hat{p})$ for $\alpha=3$ and $\hat{p}\in [0;1)$ be a sequence bounded away from one by a constant.\\
If for all $\eps=\eps(n)\to 0$
$$
\Pra{\Gnm{}{\sqrt{-\frac{\ln (1-\frac{\hat{p}}{1+\eps})}{m}}}\in\mathcal{A}}\to a
$$
and 
$$
\Pra{\Gnm{}{\sqrt{-\frac{\ln (1-\hat{p})}{(1-\eps)m}}}\in\mathcal{A}}\to a 
$$
then 
$$
\Pra{\Gn{\hat{p}}\in\mathcal{A}}\to a.
$$
\end{itemize}
\end{tw}

In (i) and (ii) for $\alpha=3$ we have to exclude the case $p = \Theta(1/n\sqrt{m})$ and $\hat{p}=\Theta(n^{-2}m^{1/3})$, since the thesis is not true on the threshold of  triangle appearance (see \cite{GpPoissonCliques}). In relation to assumptions of (ii), it should be pointed out that the case $\hat{p}(n)=o(n^{-2}m^{1/3})$ is included in Theorem~\ref{alpha6}.

The method of the proof is strong enough to show sharper results in many cases. For example, for $\alpha>3$ a function $\eps(n)$ may be replaced by $1/n^{\delta}$, where $\delta$ is a constant depending on $\alpha$. We state here two theorems as an example of how tight the results may be, if we make some additional assumptions.

\begin{tw}\label{Twierdzenie4} 
Let $a\in [0;1]$,  $\mathcal{A}$ be any monotone property ,  $m=n^{\alpha}$ for $\alpha > 4$ and
$p$ be as in \eqref{RownanieGlowne}. 
Let
\begin{align*}
\hat{p}_-&=\ \ 1-\exp(-mp^2(1-p)^{n-2});\\
\hat{p}_+&=1-\exp(-mp^2(1-p)^{n-2})+ 30 \sqrt[3]{mp^3}.
\end{align*}
If
$$
\Pr\{\Gn{\hat{p}_-}\in\mathcal{A}\}\to a\quad \text{and} \quad\Pr\{\Gn{\hat{p}_+}\in\mathcal{A}\}\to a 
$$
then 
$$
\Pr\{\Gnm{}{p}\in\mathcal{A}\}\to a.
$$
\end{tw}

\begin{tw}\label{Twierdzenie313}
Let $a\in [0;1]$, $\mathcal{A}$ be any monotone property,  $m=n^{\alpha}$ for $\alpha > 10/3$ and $p$ be as in \eqref{RownanieGlowne}. Let
\begin{align*}
\hat{p}_-=&\ \ \  1-\exp(-mp^2(1-p)^{n-2});\\
\hat{p}_+=&
\begin{cases}
1-\exp(-mp^2(1-p)^{n-2})+ 90  \sqrt[3]{mp^3},\\
\quad\quad\quad\quad\text{ for }
\Omega\left(n^{-1}m^{-1/3}\right)=p=o\left(n^{-1}m^{-1/4}\right);\\  
 1-\exp(-mp^2(1-p)^{n-2})+ 90  \sqrt[3]{mp^3} + 471 \sqrt[6]{mp^4},\\
 \quad\quad\quad\quad\text{ for }
 \Omega\left(n^{-1}m^{-1/4}\right)=p=O\left(m^{-1/2}\ln^{1/2} n\right).  
\end{cases}
\end{align*}
If
$$
\Pr\{\Gn{\hat{p}_-}\in\mathcal{A}\}\to a\quad \text{and} \quad\Pr\{\Gn{\hat{p}_+}\in\mathcal{A}\}\to a 
$$
then 
$$
\Pr\{\Gnm{}{p}\in\mathcal{A}\}\to a.
$$
\end{tw}

\section{Auxiliary definitions, inequalities and facts}\label{SectionAuxiliary}

\subsection{Coupling}
In the proofs a coupling argument is frequently used. Let $<\PP,\prec>$ be a countable partially ordered set. Usually $\PP$ stands for a subset of $\mathbb{N}$ with relation $\leq$, a Cartesian product $\mathbb{N}^t$ with relation $(x_1,\ldots,x_t)\prec (y_1,\ldots,y_t)\Leftrightarrow \forall_{1\le i\le t}x_i\le y_i$ or a set of hypergraphs $\mathcal{G}$ on a given set of vertices with relation $\subseteq$ of being a subhypergraph. In the article the set $\mathcal{G}$ is either the set of all graphs or hypergraphs on $n$ vertices or the set of $k$--partite graphs or hypergraphs with partitions with $n$ vertices. To omit unnecessary formalities it is not directly stated which partially ordered set is  considered, when it is obvious from the context. Let $X$ and $Y$ be two random variables with values in $\PP$. We write 
$$
X\coupling{q} Y,
$$
if there exists a coupling $(X,Y)$ of the random variables  
such that $X\prec Y$ with probability $q$ (i.e. if there exists a probability space $\Omega$ and two random variables $X'$ and $Y'$, such that $X'$ and $Y'$ are both defined on $\Omega$, have probability distribution as $X$ and $Y$, respectively, and $X'\prec Y'$ with probability $q$). 
We use the fact that such coupling exists if and only if there exists a probability measure $\mu: \PP\times \PP\to [0;1]$ such that for any set $A\subseteq \PP$ we have $\mu(A\times \PP)=\Pr\{X\in A\}$ and $\mu(\PP\times A)=\Pr\{Y\in A\}$ and $\mu(\{(x,y)\in \PP\times \PP ; x\prec y\})=q$). 

Now two useful facts are stated.
The simple proofs are added for completeness of considerations.

\begin{fact}\label{FaktCouplingPrzechodniosc}
Let $\PP$ be a countable partially ordered set and $X$ and $Y$ be random variables with values in $\PP$. If
\begin{equation}\label{EqFact1}
X \coupling{1-q_1} Y \quad \text{and} \quad Y\coupling{1-q_2} Z,
\end{equation}
then for some $q\le q_1 + q_2$
$$
X \coupling{1-q} Z.
$$
\end{fact}
\begin{proof}
Let $\mu_1, \mu_2: \PP\times\PP\to [0;1]$ be probability measures associated with couplings existing by~\eqref{EqFact1}. Let $\PP^*=\{y\in \PP:\Pra{Y=y}\neq 0\}$. Define 
\begin{align*}
\mu_3
&:\PP\times\PP^*\times \PP\to [0;1],
& \mu_3(x,y,z)
&=\frac{\mu_1(x,y)\mu_2(y,z)}{\Pra{Y=y}};\\
\mu
&:\PP\times\PP\to [0;1],
& \mu(x,z)
&=\mu_3(\{x\}\times \PP^* \times \{z\}).
\end{align*}
Then for $A_1=\{(x,y,z):x\prec y\}$ and $A_2=\{(x,y,z):y\prec z\}$ we have
\begin{multline*}
\mu(\{(x,z):x\prec z\})=\mu_3(\{(x,y,z):x\prec z\})\ge\\
\ge\mu_3(A_1\cap A_2)\ge \mu_3(A_1)+\mu_3(A_2)-1=1-(q_1+q_2). 
\end{multline*}
\end{proof}

\begin{fact}\label{FaktCouplingSuma}
If $(X_1,\ldots,X_t)$ and $(Y_1,\ldots,Y_t)$ are vectors of independent random variables and
\begin{equation}\label{EqFact2}
X_i \coupling{q_i} Y_i, \quad \text{for all }  1\le i \le t,
\end{equation} 
then 
$$
(X_1,\ldots,X_t) \coupling{q} (Y_1,\ldots,Y_t)
$$
and
$$
\sum_{i=1}^t X_i \coupling{q'} \sum_{i=1}^t Y_i,
$$
where
$q,q'\ge \prod_{i=1}^{k} q_i$.
\end{fact}
\begin{proof}
For all $1\le i\le t$, let $\mu_i:\PP\times \PP\to [0;1]$ be a probability measure associated with a coupling existing by $X_i \coupling{q_i} Y_i$. Simple calculation shows that $\mu:\PP^{t}\times \PP^{t}\to [0;1]$ such that 
$$
\mu(x_1,\ldots,x_t,y_1,\ldots,y_t)=\prod_{i=1}^{t}\mu_i(x_i,y_i)
$$
implies the thesis.
\end{proof}

\subsection{Total variation distance}

Let $X$ and $Y$ be random variables with values in a countable set $\PP$. We define the total variation distance between $X$ and $Y$ by 
$$
\dtv{X}{Y}=\max_{A\subseteq\PP}|\Pr\{X\in A\}-\Pr\{Y\in A\}|=\frac{1}{2}\sum_{x\in \PP}|\Pr\{X=x\}-\Pr\{Y=x\}|. 
$$
Now let $\PP=\mathcal{G}$ be a set of hypergraphs (graphs) with a given vertex set. Let $\mathbb{G}_1$ and $\mathbb{G}_2$ be two random variables with values in $\mathcal{G}$. Since
$$
\dtv{\mathbb{G}_1}{\mathbb{G}_2}=\frac{1}{2}\sum_{G\in \mathcal{G}}|\Pr\{\mathbb{G}_1=G\}-\Pr\{\mathbb{G}_2=G\}|, 
$$
it is simple to construct a probability measure $\mu$ on $\mathcal{G}\times\mathcal{G}$ with marginal distributions as distributions of $\mathbb{G}_1$ and $\mathbb{G}_2$ such that $\mu\{(G,G):G\in \mathcal{G}\}=1-2\dtv{\mathbb{G}_1}{\mathbb{G}_2}$. This implies:
\begin{fact}\label{FaktDtvCoupling}
$$
\mathbb{G}_1\coupling{q}\mathbb{G}_2\quad\text{and}\quad\mathbb{G}_2\coupling{q'}\mathbb{G}_1,
$$
where $q,q'\ge 1-2\dtv{\mathbb{G}_1}{\mathbb{G}_2}$. 
\end{fact}

The following useful facts concerning total variation distance are Facts~3 and~4 in \cite{GpEquivalence}.
\begin{fact}\label{FaktDtvWarunkowe}
 Let $A$ and $A'$ be random variables with values in the same set. If there exist random variables $B$ and $B'$ such that for all possible $b$ the distribution of $A$ under condition $B=b$ and the distribution of $A'$ under condition $B'=b$ are the same, then
$$
\dtv{A}{A'}\le 2\dtv{B}{B'}.
$$
\end{fact}
\begin{fact}\label{FaktDtvRoznica}
Let $A$ and $A'$ be two random variables. If there exists a probability space on which random variables $B$ and $B'$ are both defined and have probability distribution as $A$ and $A'$, respectively, then
$$
\dtv{A}{A'}\le\Pr\{B\neq B'\}.
$$ 
\end{fact}
\noindent
We also use a standard result (see for example~\cite{KsiazkaPoisson} equation (1.23)).
\begin{fact}\label{FaktDtvPoisson}
Let $A$ be a random variable with the binomial distribution $\Bin(\hat{n},\hat{p})$ and let $A'$ be a random variable with the Poisson distribution $\Po(\hat{n}\hat{p})$. Then
$$
\dtv{A}{A'}\le \hat{p}.
$$ 
\end{fact}

\subsection{Coupon collector model}\label{SubsectionCouponCollector}

We define two auxiliary random variables, which are generalised versions of random variables defined in \cite{GpEquivalence}. Let $K\ge 2$ be a given constant integer, $M$ be any random variable with values in $\mathbb{N}$, $\overline{n}=(n_2,\ldots,n_K)$ be a vector of positive integers and $\overline{P}=(P_2,\ldots,P_K)$ be a vector of nonnegative reals such that $\sum_{k=2}^K n_kP_k\le 1$. Assume now that we have $\sum_{k=2}^K n_k$ coupons
$\bigcup_{k=2}^{K}\{c^{(k)}_{1}, \ldots ,c^{(k)}_{n_k}\}$ and one blank coupon $d_0$. We make $M$ independent draws, with replacement, such that in each draw
\begin{align*}
\Pr\{c^{(k)}_{i}\text{ is chosen}\}&=P_k,\quad\text{for }2\le k\le K, 1\le i\le n_k;\\
\Pr\{d_0\text{ is chosen}\}&=1-\sum_{k=2}^K n_kP_k.
\end{align*}
In this scheme we define $\kr{k}{i}(M)$ to be a random variable denoting the number of times that a coupon   $c^{(k)}_{i}$ was chosen and 
\begin{align*}
\kra{k}{i}(M)&=
\begin{cases}
1&\text{ if } \kr{k}{i}(M)\ge 1;\\
0&\text{ otherwise}. 
\end{cases}\\
\end{align*}
The first auxiliary random variable is
\begin{equation}\label{RownanieDefinicjaXM}
X(M)=X(\overline{n},\overline{P},M)=(\kra{2}{}(M),\ldots,\kra{K}{}(M)), \text{ where }
\end{equation}
$$
\kra{k}{}(M)=\sum_{i=1}^{n_k}\kra{k}{i}(M).
$$
The second random variable is
\begin{equation}\label{RownanieDefinicjaY}
Y=Y(\overline{n},\overline{P'})=(\kraw{2}{},\ldots,\kraw{K}{})
\end{equation}
where $\overline{P'}=(P'_2,\ldots,P'_K)$ is a vector such that $P_k'\le 1$ for all $2\le k\le K$ and $\kraw{k}{}$, $2\le k\le K$, are independent random variables with the binomial distribution $\Bin(n_k,P_k')$.

A simple observation stated below is a generalisation of a part of the proof of Claim~1~in~\cite{GpEquivalence} and may be shown by careful calculation.  

\begin{fact}\label{FaktXMPoisson}
Let $M$ be a random variable with the Poisson distribution $\Po(\lambda)$, then $\kr{k}{i}(M)$, $2\le k\le K$ and $1\le i\le n_k$, are independent random variables with the Poisson distribution $\Po(\lambda P_k)$. Moreover $\kra{k}{}(M)$, $2\le k\le K$, are independent random variables with the binomial distribution $\Bin(n_k, 1-\exp(-\lambda P_k))$. Therefore $X(M)$ and $Y$ have the same distribution for $P'_k=1-\exp(-\lambda P_k)$. 
\end{fact}
\noindent It is also simple to show the following fact.
\begin{fact}\label{FactCouplingWarunkowe}
Let $M$ and $M'$ be random variables with values in $\mathbb{N}$. If 
$$
M\coup{o(1)}M',
$$
then
$$
X(M)\coup{o(1)}X(M').
$$
\end{fact}

\subsection{Chernoff's bound}
For the proofs of Chernoff's bound see Theorem~2.1 in \cite{KsiazkaJLR}.
\begin{lem}
Let $X$ be a random variable with the binomial distribution and $\lambda=\E X$. Let $a \ge \lambda$, then
$$
\Pra{X\ge a}\le \exp\left(-\lambda-a\ln \frac{a}{\lambda} + a\right)
$$
\end{lem}
\noindent After careful calculation we obtain the following lemma.
\begin{lem}\label{RownanieChernoffDokladny1}
Let $t\ge 1$ be an integer and $X_n$ be a sequence of random variables with the binomial distribution, such that $\E X_n=\lambda_n$. Let $\eps>0$ and $\omega(n)$  be any function tending to infinity. If
\begin{equation}
a_n=a_n(\lambda_n,t,\eps)=
\begin{cases}
(t+\eps) \ln n / (\ln \ln n - \ln \lambda_n),&\text{ for }\lambda_{n}=o(\ln n);\\
\omega(n)\lambda_n,&\text{ for } \lambda_n = \Theta(\ln n);\\
(1+\eps)\lambda_n,&\text{ for }\ln n= o (\lambda_n),
\end{cases}
\end{equation}
then
$$
\Pra{X_n\ge a_n}=o\left(n^{-t}\right).
$$
\end{lem}

\begin{lem}\label{RownanieChernoffDokladny}
Let $X_n$ be a sequence of random variables with the binomial distribution. Then 
\begin{equation}\label{RownanieChernoff}
\begin{split}
&\Pr\left\{X_n \le \E  X_n - t_n\right\}\le \exp\left(-\frac{t_n^2}{2\E X_n}\right),
\ \ \ \ \ \ \ \ \  \text{ for }t_n\ge 0;\\
&\Pr\left\{X_n \ge \E X_n + t_n\right\}\le \exp\left(-\frac{3t_n^2}{2(3\E X_n + t_n)}\right),
\text{ for }t_n\ge 0.
\end{split}
\end{equation}
\end{lem}
\noindent It is also possible to formulate the version of Chernoff's bound for random variables with the Poisson distribution.
\begin{lem}
Let $X_n$ be a sequence of random variables with the Poisson distribution $\Po(\lambda)$ and $i>0$ be any constant, then 
\begin{equation}\label{RownanieChernoffPoisson}
\begin{split}
&\Pr\left\{X_n \le \E X_n - t_n\right\}\le \exp\left(-\frac{t_n^2}{2\E X_n}\right)+o\left(\frac{1}{n^{i}}\right),
\ \ \ \ \ \ \ \ \  \text{ for }t_n\ge 0;\\
&\Pr\left\{X_n \ge \E X_n + t_n\right\}\le \exp\left(-\frac{3t_n^2}{2(3\E X_n + t_n)}\right)+o\left(\frac{1}{n^{i}}\right),
\text{ for }t_n\ge 0.
\end{split}
\end{equation}
\end{lem}
\begin{proof}
It follows by \eqref{RownanieChernoff} applied to random variable with the binomial distribution $\Bin(\lambda n^{i+1},1/n^{i+1})$, definition of the total variation distance and Fact~\ref{FaktDtvPoisson}.
\end{proof}

\section{Coupling of $\Gnm{}{p}$ and $\Gn{\hat{p}}$}\label{SectionCoupling}

\subsection{Relation between $\Hnm{k}{p}$ and $\Hn{k}{1-\exp(-mp^k(1-p)^{n-k})}$}

As it is pointed out in Introduction, in the proof $\Gnm{}{p}$ is related to $\Gn{\hat{p}}$ trough $\Gnm{k}{p}=G\Hnm{k}{p}$ and $G\Hn{k}{1-\exp(-mp^k(1-p)^{n-k})}$. In the subsection a lemma, which shows relations between
$\Hnm{k}{p}$ and $\Hn{k}{1-\exp(-mp^k(1-p)^{n-k})}$, is proved.

\begin{lem}\label{LematDtv}
Let $K\ge 2$ be a constant integer and $p=o(1/n)$, then
$$
\dtv{\bigcup_{k=2}^{K}\Hnm{k}{p}}{\bigcup_{k=2}^{K}\Hn{k}{1-\exp(-mp^k(1-p)^{n-k})}}=o(1),
$$  
where $\Hn{k}{1-\exp(-mp^k(1-p)^{n-k})}$ are independent random hypergraphs.
\end{lem}

Let, for all $2\le k\le K$,
$$
p_k=p^k(1-p)^{n-k},\quad  n_k=\binom{n}{k}, \quad P_k=\frac{p_k}{\sum_{k=2}^{K}p_kn_k} \quad\text{and}\quad P_k'=1-\exp(-mp_k) 
,$$
$M$ have the binomial distribution $\Bin(m,P)$ ($P=\sum_{k=2}^{K}p_kn_k$), $X(M)$ be defined as in \eqref{RownanieDefinicjaXM} and $Y$ be defined as in \eqref{RownanieDefinicjaY}. Then for all $2\le k\le K$
\begin{align*}
|E(\Hnm{k}{p})|=\kra{k}{}(M)\quad \text{and}\quad  |E(\Hn{k}{1-\exp(-mp_k)})|=\kraw{k}{}. 
\end{align*}
Moreover for any two 
hypergraphs $H$ and $H'$, such that for all $k\ge 2$ the number of edges of cardinality $k$ in $H$ and $H'$ is the same, we have
$$
\Pra{\bigcup_{k=2}^{K}\Hnm{k}{p}=H}=\Pra{\bigcup_{k=2}^{K}\Hnm{k}{p}=H'}
$$
and
$$
\Pra{\bigcup_{k=2}^{K}\Hn{k}{1-\exp(-mp_k)}=H}=\Pra{\bigcup_{k=2}^{K}\Hn{k}{1-\exp(-mp_k)}=H'}
.
$$ 
Therefore, by Fact~\ref{FaktDtvWarunkowe} a following lemma implies Lemma~\ref{LematDtv}. \begin{lem}\label{LemmaClaim1}
Let $K\ge 2$ be a constant integer.
Let $p=o(1/n)$, $M$ be a random variable with the binomial distribution $\Bin(m,\sum_{k=2}^K\binom{n}{k}p_k)$, $p_k=p^{k}(1-p)^{n-k}$,  $X(M)$ be defined as in \eqref{RownanieDefinicjaXM} for $n_k=\binom{n}{k}$ and $P_k=p_k/(\sum_{k=2}^K p_kn_k)$. Let moreover $Y$ be defined as in \eqref{RownanieDefinicjaY} for $P'_k=1-\exp(-mp_k)$. Then
$$
\dtv{X(M)}{Y}=o(1).
$$
\end{lem}
\noindent In fact, Lemma~\ref{LemmaClaim1} is a stronger and more general version of Claim~1 from~\cite{GpEquivalence}.
In the proof the main idea of the proof of Claim~1 from~\cite{GpEquivalence} is used. However, a modification of the choice of $M$ and $M'$ enables us to extend the result for $\alpha\le 4$.
\begin{proof}

We replace the binomial random variable $M$ with a Poisson random variable $M'$ with the same expected value $m\sum_{k=2}^K\binom{n}{k}p_k$. 
By Fact~\ref{FaktXMPoisson} the random variables $\kra{j}{i}(M')$ are independent and
$$
\Pr\{\kra{k}{i}(M')=1\}=1-\exp(-mp_k).
$$ 
Therefore in
$X(M')=(\kra{2}{}(M'),\ldots,\kra{K}{}(M'))$, the random variables $\kra{2}{}(M'),\ldots, \kra{K}{}(M')$ are independent with the binomial distribution $\Bin(n_2,1-\exp(-mp_2)),\ldots,\Bin(n_K,1-\exp(-mp_K))$, respectively.
By definition of $Y$, Facts~\ref{FaktDtvWarunkowe} and \ref{FaktDtvPoisson} we have
\begin{multline*}
\dtv{Y}{X(M)}=\dtv{X(M')}{X(M)}\le 2\dtv{M'}{M}\le\\
\le 2\sum_{k=2}^K\binom{n}{k}p_k=O\left(\sum_{k=2}^Kn^kp^k\right)=o(1).
\end{multline*}
\end{proof}

\subsection{Couplings of $G\Hn{k}{q}$ and a graph with independent edges}

In view of lemma shown in previous subsection there is a relation between 
$\bigcup_{k=2}^{K}\Gnm{k}{p}$ 
and
 $\bigcup_{k=2}^{K}G\Hn{k}{1-\exp(-mp^k(1-p)^{n-k})}$. The second importand part of the proof of the main theorems is to relate  $\bigcup_{k=2}^{K}G\Hn{k}{1-\exp(-mp^k(1-p)^{n-k})}$ to a graph with independent edges.

\noindent Let
\begin{equation}\label{RownanieStala}
a_n(q)=
\begin{cases}
6 &\text{ for }nq^2=O(n^{-1/2}) ;\\
3\ln n/(-\ln nq^2 + \ln \ln n)&\text{ for }nq^2=o(1)\text{ and }o(nq^2)=n^{-1/2};\\
3\ln n/\ln \ln n &\text{ for }nq^2=\Theta(1);\\
3\ln n/(\ln \ln n - 3\ln nq^2) &\text{ for }nq^2\to \infty \text{ and }nq^2=o(\sqrt[3]{\ln n});\\
\omega(n)n^3q^6,\text{ where }\omega(n)\to \infty&\text{ for }nq^2\to \infty \text{ and }nq^2=\Theta(\sqrt[3]{\ln n})\\
cn^3q^6,\text{ where }c>1&\text{ for }nq^2\to \infty \text{ and }o(nq^2)=\sqrt[3]{\ln n};\\
\end{cases}
\end{equation}
In this subsection three following lemmas are proved. 

\begin{lem}\label{LematCoupling3}
Let $c_3>2\cdot 3/\sqrt[3]{3!}$, $q=\Omega(n^{-1})$ and $q=o(n^{-3/7})$.
$$
G\Hn{3}{q^{3}}\coup{o(1)} \Gn{a_n(c_3q)c_3q},
$$ 
where $a_n(q)$ is defined as in \eqref{RownanieStala}.
\end{lem}
\begin{lem}\label{LematCoupling4}
Let $c_4>\sqrt[3]{15}\cdot 3\cdot 4/\sqrt[6]{4!}$, $q=\Omega(n^{-2/3})$ and $q=o(n^{-3/7})$.
$$
G\Hn{4}{q^{6}}\coup{o(1)} \Gn{a_n(c_4q)c_4q},
$$ 
where $a_n(q)$ is defined as in \eqref{RownanieStala}.
\end{lem}
\begin{lem}\label{LematCoupling5}
Let $c_5>\sqrt[6]{2^2\cdot 3\cdot 5^3}\cdot4\cdot 5/\sqrt[10]{5!}$, $q=\Omega(n^{-1/2})$ and $q=o(n^{-3/7})$.
$$
G\Hn{5}{q^{10}}\coup{o(1)} \Gn{a_n(c_5q)c_5q},
$$ 
where $a_n(q)$ is defined as in \eqref{RownanieStala}.
\end{lem}

The following fact shows that the problem reduces to a $k$--partite case. 
First let us introduce additional notation.
Let $\X{1},\ldots,\X{k}$ be disjoint $n$-element sets and $r\in[0;1]$. We define $\Hk{k}{r}$  to be a hypergraph with vertex set $\bigcup_{i=1}^{k}\X{k}$ and edge set being the random subset of $\mathcal{E}:=\{(x_1,\ldots,x_k):\forall_{1\le i\le k}\ x_i\in\X{i}\}$ such that each element from $\mathcal{E}$ is added to  $E(\Hk{k}{r})$ independently with probability $r$. Let moreover $\Gkn{k}{r}$ be a random $k$--partite graph with $k$-partition $(\X{1},\ldots,\X{k})$ and each edge appearing with probability $r$. 

\begin{fact}
Let $a_n=\Omega(1)$. If 
\begin{equation}\label{RownanieCouplingKdzielne}
G\Hk{k}{r^{\binom{k}{2}}}\coup{o(1)} \Gkn{k}{a_nr},
\end{equation}
then,
$$
G\Hn{k}{1-(1-r^{\binom{k}{2}})^{k!}}\coup{o(1)} \Gn{1-(1-a_nr)^{k(k-1)}}
$$
and if $a_n r=o(1)$, then for any constant $c>k(k-1)/(k!)^{1/\binom{k}{2}}$
$$
G\Hn{k}{r^{\binom{k}{2}}}\coup{o(1)} \Gn{c\: a_nr}.
$$
\end{fact}
\begin{proof}
Let $\X{i}=\{x^{(i)}_1,\ldots,x^{(i)}_n\}$, for $1\le i\le k$, and $\V=\{v_1,\ldots,v_n\}$.
For a given instance of $\Hk{k}{r^{\binom{k}{2}}}$ (or $\Gkn{k}{a_nr}$) one may construct an instance of a hypergraph $\Hn{k}{1-(1-r^{\binom{k}{2}})}$ (or a graph $\Gn{1-(1-a_nr)^{k(k-1)}}$) with vertex set $\V$ by merging all vertices $x^{(i)}_j$, $1\le i\le k$, into $v_j$, for all $1\le j\le n$ , and deleting edges with less then $k$ (or $2$) vertices.
\end{proof}

Therefore three following lemmas imply Lemmas~\ref{LematCoupling3}, \ref{LematCoupling4} and \ref{LematCoupling5}.

\begin{lem}\label{TwierdzenieCoupling3}
Let $q=\Omega(n^{-1})$ and $q=o(n^{-1/3})$.
$$
G\Hk{3}{q^{3}}\coup{o(1)} \Gkn{3}{a_n(q)q},
$$ 
where $a_n(q)$ is defined as in \eqref{RownanieStala}.
\end{lem}
The above lemma is a generalisation of Theorem~1.7 from \cite{Coupling}, where it was stated for $(\ln n/n^2)^{1/3}=o(q), q= o(n^{-3/5})$ and $a_n=17$. 
\begin{lem}\label{TwierdzenieCoupling4}
Let $q=\Omega(n^{-2/3})$, $c'_4>\sqrt[3]{15}$ and $q=o(n^{-2/5})$.
$$
G\Hk{4}{q^{6}}\coup{o(1)} \Gkn{4}{a_n(c'_4q)c'_4q},
$$ 
where $a_n(q)$ is defined as in \eqref{RownanieStala}.
\end{lem}
\begin{lem}\label{TwierdzenieCoupling5}
Let $q=\Omega(n^{-1/2})$, $c'_5>\sqrt[6]{2^2\cdot 3\cdot 5^3}$ and $q=o(n^{-2/5})$.
$$
G\Hk{5}{q^{10}}\coup{o(1)} \Gkn{5}{a_n(c'_5q)c'_5q},
$$ 
where $a_n(q)$ is defined as in \eqref{RownanieStala}.
\end{lem}

\noindent For clarity of considerations long proofs of Lemmas~\ref{TwierdzenieCoupling3}, \ref{TwierdzenieCoupling4} and \ref{TwierdzenieCoupling5}
are left to Appendix.

\subsection{Main coupling lemma}

\begin{lem}\label{LematGlowny}
Let $a_n(q)$ be defined as in \eqref{RownanieStala}.
Moreover let $c_3>2\cdot 3/\sqrt[3]{3!}$, $c_4>\sqrt[3]{15}\cdot3\cdot 4/\sqrt[6]{4!}$, $c_5>\sqrt[6]{2^2\cdot 3\cdot 5^3}\cdot 4\cdot 5/\sqrt[10]{5!}$, $q_k=(1-\exp(-mp^k(1-p)^{n-k}))^{1/\binom{k}{2}}$, for $k=2,3,4,5$ and 
\begin{align}
\label{RownaniePminus}
\hat{p}_-=&\ \ \  q_2\\
\label{RownaniePplus}
\hat{p}_+=&
\begin{cases}
q_2+a_n(c_3 q_3)c_3 q_3,
&\text{for }
p=\Omega(n^{-1}m^{-1/3})\text{ and }
\\
&\quad \ \: p=o(min\{n^{-1}m^{-1/4},n^{-3/7}m^{-1/3}\});\\
%%%%%%%%%%%%%%%%%%%%%%%%%%%%%%%%%%%%%%%%%%%%%%%%%%%%%%5
q_2+\sum_{k=3}^4a_n(c_k q_k)c_k q_k,
&\text{for }
p=\Omega(n^{-1}m^{-1/4})\text{ and }
\\
&\quad \ \: p=o(min\{n^{-1}m^{-1/5},n^{-3/7}m^{-1/3},\\
&\quad \ \: \ \ \ \ \ \ \ n^{-9/14}m^{-1/4}\});\\
%%%%%%%%%%%%%%%%%%%%%%%%%%%%%%%%%%%%%%%%%%%%%%%%%%%%%%5
q_2+\sum_{k=3}^5 a_n(c_k q_k)c_k q_k,
&\text{for }
p=\Omega(n^{-1}m^{-1/5})\text{ and }
\\
&\quad \ \: p=o(min\{n^{-1}m^{-1/6},n^{-3/7}m^{-1/3},\\
&\quad \ \: \ \ \ \ \ \ \ n^{-9/14}m^{-1/4},n^{-6/7}m^{-1/5}\}).\\
%%%%%%%%%%%%%%%%%%%%%%%%%%%%%%%%%%%%%%%%%%%%%%%%%%%%%%5
\end{cases}
\end{align}
Then
\begin{equation*}\label{RownanieLematGlowny}
\Gn{\hat{p}_-}\coup{o(1)}\Gnm{}{p}\quad\text{and}\quad\Gnm{}{p}\coup{o(1)}\Gn{\hat{p}_+}.
\end{equation*}

\end{lem}
\begin{proof}
In the statement of the lemma we have 3 different values of $\hat{p}_+$. They correspond to three cases:  $\Hnm{4}{p}$ is empty with high probability,
$\Hnm{5}{p}$ is empty with high probability,
$\Hnm{6}{p}$ is empty with high probability. 
We prove Lemma~\ref{LematGlowny} in all three cases at the same time. The proof differs only by the value of $K$, which is $3,4$ and $5$ in the first, second and third case, respectively.
  
Let $m=n^{\alpha}$ and $q_k=(1-\exp(-mp^{k}(1-p)^{n-k}))^{1/\binom{k}{2}}$.
We prove that under assumptions of Lemma~\ref{LematGlowny} there exists a sequence of couplings

\begin{align}
\Gn{q_2}&\coup{o(1)}&&\label{coupling1} \\ 
\Gnm{2}{p}&
\coupling{1}
\ \Gnm{}{p}&&
\coup{o(1)}
\bigcup_{k=2}^{K}\Gnm{k}{p} \label{coupling23}
\\ \label{coupling4}
&&&\coup{o(1)}
\bigcup_{k=2}^{K}G\Hn{k}{q_k^{\binom{k}{2}}}
\\ \label{coupling5}
&&&\coup{o(1)}
\Gn{q_2}\cup\left(\bigcup_{k=2}^{K}\Gn{a_n(c_kq_k)c_kq_k}\right) 
\\ \label{coupling6}
&&&\coupling{1}
\Gn{q_2+\sum_{k=3}^{K}a_n(c_kq_k)c_kq_k}. 
\end{align}
Here  $$G\Hn{2}{q_2^{\binom{2}{2}}},\ldots,G\Hn{K}{q_K^{\binom{K}{2}}}$$ are independent random hypergraphs.

Couplings \eqref{coupling1} and \eqref{coupling4} follow by Lemma~\ref{LematDtv} and Fact~\ref{FaktDtvCoupling}.
The left--hand side of \eqref{coupling23} is trivial. A coupling existing by the right--hand side of \eqref{coupling23} follows by the fact that under the assumptions of Lemma~\ref{LematGlowny}
\begin{align*}
\Pr\{\exists_{w\in \W}\: |V(v)| > K\}
=O\left(mn^{K+1}p^{K+1}\right)=o(1)
.
\end{align*}
Moreover
\eqref{coupling5} is a consequence of Lemma~\ref{LematCoupling3}, \ref{LematCoupling4} and \ref{LematCoupling5}  after substituting $q=q_k$ for $k=3,\ldots,K$.
Finally coupling from \eqref{coupling6} is standard. Therefore the lemma follows by Fact~\ref{FaktCouplingPrzechodniosc}.
\end{proof}

\section{Proof of the theorems}\label{SectionOutline}

The proof of Theorem~\ref{alpha6} uses similar techniques to those of the proof presented in~\cite{GpEquivalence}.

\begin{proof}[Proof of Theorem~\ref{alpha6}]
For $p=o(1/n\sqrt[3]{m})$ by Fact~\ref{FaktDtvRoznica} and Lemma~\ref{LematDtv} with $K=2$ we have  
\begin{align*}
&\dtv{\Gnm{}{p}}{\Gn{\hat{p}}}\le\\
&\dtv{\Gnm{}{p}}{\Gnm{2}{p}}+\dtv{\Gnm{2}{p}}{\Gn{\hat{p}}}\le\\
&\le \Pra{\Gnm{}{p}\neq \Gnm{2}{p}} + \dtv{\Gnm{2}{p}}{\Gn{\hat{p}}}\le\\
&\le \Pra{\exists_{w\in \W}|V(w)|>2} + \dtv{\Gnm{2}{p}}{\Gn{\hat{p}}}\le\\
&\le m\binom{n}{3}p^3+\dtv{\Gnm{2}{p}}{\Gn{\hat{p}}}=o(1).
\end{align*}
\end{proof}

\noindent The proofs of Theorems \ref{Twierdzenie3}, \ref{Twierdzenie4} and \ref{Twierdzenie313} base on the following fact.

\begin{fact}\label{FaktGplusminus}
Let $G_-$, $G$ and $G_+$ be random graphs such that
\begin{equation}\label{RownanieGplusminus}
G_-\coup{o(1)} G\quad\text{and}\quad G\coup{o(1)} G_+.
\end{equation}
If for $a\in[0;1]$ and a monotone property $\mathcal{A}$
\begin{equation}\label{RownaniePrawdGplusminus}
\Pra{G_-\in \mathcal{A}}\to a\quad\text{and}\quad  \Pra{G_+\in \mathcal{A}}\to a.
\end{equation}
then  
$$
\Pra{G\in \mathcal{A}}\to a.
$$
\end{fact}
\begin{proof} By \eqref{RownanieGplusminus} there exists a probability space on which we may define random vectors  $(G_-,G)$ and $(G,G_+)$ such that
$$
\Pra{\mathcal{E}_-}=1-o(1)\quad\text{and}\quad\Pra{\mathcal{E}_+}=1-o(1)),
$$
for events
$$
\mathcal{E}_-:=\{G_-\subseteq G\}
\quad\text{and}\quad
\mathcal{E}_+:=\{G \subseteq G_+\}. 
$$
If \eqref{RownaniePrawdGplusminus} then on the probability space
\begin{align*}
\Pr\{ G\in \mathcal{A}\}
&\le\Pr\{ G\in \mathcal{A}|\mathcal{E}_+\}\Pr\{\mathcal{E}_+\}+\Pr\{\mathcal{E}_+^c\}\le
\\
&\le\Pr\{ G_+\in \mathcal{A}|\mathcal{E}_+\}\Pr\{\mathcal{E}_+\}+\Pr\{\mathcal{E}_+^c\}\le
\\
&\le\Pr\{ \{G_+\in \mathcal{A}\}\cap\mathcal{E}_+\}+\Pr\{\mathcal{E}_+^c\}\le
\\
&\le\Pr\{ G_+\in \mathcal{A}\}+\Pr\{\mathcal{E}_+^c\}=\\
&=\Pr\{ G_+\in \mathcal{A}\}+o(1)=a+o(1)
\intertext{and}
\Pr\{ G\in \mathcal{A}\}
&\ge\Pr\{ G\in \mathcal{A}|\mathcal{E}_-\}\Pr\{\mathcal{E}_-\}\ge
\\
&\ge\Pr\{ G_-\in \mathcal{A}|\mathcal{E}_-\}\Pr\{\mathcal{E}_-\}=
\\
&=\Pr\{ \{G_-\in \mathcal{A}\}\cap\mathcal{E}_-\}=
\\
&\ge\Pr\{ G_-\in \mathcal{A}\}+\Pr\{\mathcal{E}_-\}-\Pra{\{G_-\in \mathcal{A}\}\cup\mathcal{E}_-}\ge\\
&\ge \Pr\{ G_-\in \mathcal{A}\}+\Pr\{\mathcal{E}_-\}-1=\\
&=\Pr\{ G_-\in \mathcal{A}\}+o(1)=a+o(1).
\end{align*}
Analogous equalities may be formulated for a decreasing property.
\end{proof}

\begin{proof}[Proof of Theorem~\ref{Twierdzenie3}]
\ \\
(i) 
By Lemma~\ref{LematGlowny} and Fact~\ref{FaktGplusminus} in order to prove Theorem~\ref{Twierdzenie3}(i) it remains to show that 
$$
\hat{p}_+\le (1+\eps'(n))q_2 \quad\text{for some function }\eps'(n)\to 0,
$$
where $\hat{p}_+$ and $q_2$ are defined as in the statement of Lemma~\ref{LematGlowny}. 
For completeness it should be pointed out that under assumptions of Theorem~\ref{Twierdzenie3}(i) $p$ fulfils all the conditions from \eqref{RownaniePplus}. 

By \eqref{RownaniePplus} we are reduced to proving that for $k=3,4,5$
\begin{equation}\label{RownanieDowodTw}
\frac{a_n(c_kq_k)c_kq_k}{q_2}=o(1)\quad\text{ for }p=\Omega(n^{-1}m^{-1/k}) 
\end{equation}
Notice that
$$
\frac{a_n(c_kq_k)c_kq_k}{q_2}=
\begin{cases}
O\left(q_k\right)&\text{ for }nq_k=o(n^{-1/2});\\
O\left(q_k\ln n \right)&\text{ for }nq_k=o(\ln^{-1/3} n);\\
O\left(\omega(n)n^3q_k^7\right)&\text{ for }nq_k=\Omega(\ln^{-1/3} n)\\
&\text{ and } \omega(n) \text{ tending slowly to $0$}, 
\end{cases}
$$
$$
q_2\sim mp^2\text{ or } q_2=\Theta(1),\quad 
q_3\sim m^{1/3}p,\quad 
q_4\sim m^{1/6}p^{2/3}
\quad \text{and}\quad 
q_5\sim m^{1/10}p^{1/2}. 
$$
Moreover
$$
nq_3^2=\Omega(n^{-1/2}) \Leftrightarrow p=\Omega(n^{-3/4}m^{-1/3})
$$
and in the considered case
$$
p=\Omega(n^{-1}m^{-1/k})\quad\text{and}\quad p=O\left(\ln^{1/2}m^{-1/2}\right).
$$
If we substitute above values to
$$
\frac{a_n(c_kq_k)c_kq_n}{q_2}
$$
after a simple calculation we arrive at \eqref{RownanieDowodTw}.
\bigskip

\noindent(ii) Let 
$
p=\sqrt{-\ln (1-\hat{p})/((1-\eps')m)}
$
and $\eps'=\eps'(n)$ be such that $(1-p)^{n-2}\ge 1-\eps'$ and $\eps'=o(1)$. Since under assumptions of (ii) $\ln (1-\hat{p})=O(1)$, such $\eps'$ exists. Then by a simple calculation we have $\hat{p}\le q_2$, where $q_2$ is defined as in Lemma~\ref{LematGlowny}. Thus by Lemma~\ref{LematGlowny} and a standard coupling of $\Gn{\cdot}$ 
\begin{equation}\label{RownanieDowod(ii)1}
\Gn{\hat{p}}\coupling{1}\Gn{q_2}\coup{o(1)}\Gnm{}{\sqrt{-\frac{\ln (1-\hat{p})}{(1-\eps')m}}}.
\end{equation}

Let now 
$
p=\sqrt{-\ln (1-(\hat{p}/(1+\eps'')))/m)}, 
$
then $q_2\le \hat{p}/(1+\eps'')$, where $q_2$ is defined as in Lemma~\ref{LematGlowny}.
Under assumptions of (ii) $p$ fulfils \eqref{RownanieGlowne}, therefore by the proof of  (i) $\hat{p}_+=q_2(1+o(1))$, where $\hat{p}_+$ is defined as in Lemma~\ref{LematGlowny}. A carefull insight into the proof of (i) lead us to the conclusion that $\eps''=\eps''(n)$ may be chosen such that $\hat{p}_+\le (1+\eps'')q_2$ and $\eps''=o(1)$. Then $\hat{p}_+\le \hat{p}$ and by Lemma~\ref{LematGlowny}
\begin{equation}\label{RownanieDowod(ii)2}
\Gnm{}{\sqrt{-\frac{\ln (1-\frac{\hat{p}}{1+\eps''})}{m}}}
\coup{o(1)}
\Gn{\hat{p}_+}\coupling{1}\Gn{\hat{p}}.
\end{equation}
Therefore  \eqref{RownanieDowod(ii)1} and \eqref{RownanieDowod(ii)2} combined with Fact~\ref{FaktGplusminus} imply the thesis of (ii).
\end{proof}

\begin{proof}[Proof of Theorems~\ref{Twierdzenie4} and~\ref{Twierdzenie313} ]
The proofs of Theorems~\ref{Twierdzenie313} and \ref{Twierdzenie4} are basically the same as this of Theorem~\ref{Twierdzenie3}(i). First notice that $q_k\sim \sqrt[\binom{k}{2}]{mp^k}$ for $k=3,4$. Moreover, if we substitute $p=O(\ln^{1/2} n/ m^{1/2})$ then $a_n(c_3q_3)c_3 < 30$ for $\alpha>4$ and $a_n(c_3q_3)c_3 < 90$, $a_n(c_4q_4)c_4 < 471$ for $\alpha>10/3$. Therefore Lemma~\ref{LematGlowny} and Fact~\ref{FaktGplusminus} imply the thesis.  
\end{proof}

Notice that although the expected number of hyperedges in $G\Hn{k}{q^{\binom{k}{2}}}$ and cliques in $\Gn{q}$ is the same, the function $a_n$ is necessary. 
There exists a coupling of two random graph models, the existence of which contradicts the thesis that for some constant $C$ and for all $q$ 
$$
G\Hn{3}{q^{3}}\coup{o(1)} \Gn{Cq}.
$$
Let $q=o(1)$. For any $e$, a $3$--element subset of $\V$, define $F_e$ to be the set of bijections assigning to the numbers from the set  $\{1,2,3\}$ the vertices of $e$  ($|F_e|=6$). Now, to each $e$, a $3$-element subset of $\V$, and each function $f\in F_e$ we assign $f$ to $e$ independently of all other functions and sets with probability 
$$r=1-(1-q^{3})^{1/6}\sim \frac{q^{3}}{6}.$$
Notice that if we add each edge $e$ to the set of edges of the hypergraph with vertex set $\V$ in the case when at least one function from $F_e$ is assigned to $e$, we get a random variable with the same distribution as $\Hn{3}{q^{3}}$. Moreover we may construct a random subgraph $G_3$ of $G\Hn{3}{q^{3}}$ by adding an edge $(v_1,v_2)$, $v_1,v_2\in \V$, if and only if at least one $3$-element subset of $\V$ containing $v_1$ and $v_2$ is assigned a function in which $v_1$ and $v_2$ are assigned $1$ and $2$ or $2$ and $1$. Notice that, from independent choice of the functions from $F_e$ we get that each edge appears in $G_3$ independently with probability 
$$r'=1-(1-r)^{2(n-2)}\sim 2nr\sim \frac{1}{3}nq^{3}.$$
Therefore
$$
\Gn{r'}\coupling{1}G\Hn{3}{q^{3}}
$$  
and in the lemmas there should be $a_n=\Omega(nq^{2})$.

\section*{Appendix}

We prove Lemma~\ref{TwierdzenieCoupling3} in detail. The proof of Lemmas~\ref{TwierdzenieCoupling4} and~\ref{TwierdzenieCoupling5} are analogous, therefore we only sketch them.

\begin{proof}[Proof of Lemma~\ref{TwierdzenieCoupling3}]
For $x\in\X{3}$, let $H(x)$ be subhypergraph of $\Hk{3}{q^{3}}$ induced on $\{x\}\cup\X{1}\cup\X{2}$ (i.e. a hypergraph with vertex set $\{x\}\cup\X{1}\cup\X{2}$ and edge set consisting of those edges from $E(\Hk{3}{q^{3}})$, which contain $x$). Moreover let us denote by $H^*(x)$ a subgraph of $GH(x)$ induced on $\X{1}\cup\X{2}$. By above definitions 
\begin{equation}\label{RownanieH3sumaHx}
\Hk{3}{q^{3}}=\bigcup_{x\in\X{3}}H(x),
\end{equation}
and edges in $H(x)$ and $H^*(x)$ are independent (i.e. $H^*(x)$ and $\Hk{2}{q^{3}}$ are the same models).

Moreover we define $T(x)$, $x\in\X{3}$, to be a graph with vertex set $\{x\}\cup\X{1}\cup\X{2}$ and edge set constructed by the following procedure. First we add each edge $(x,y)$, $y\in\X{1}\cup \X{2}$ independently  with probability $C q$ to the edge set, where 
\begin{equation}\label{RownanieDefinicjaC}
C=C(q)=
\begin{cases}
c,\text{ where } c>5,&\text{ for } nq^2=o(1);\\
\omega(n),\text{ where } \omega(n)\to \infty,&\text{ for } nq^2=\Theta(1);\\
cnq^{2},\text{ where } c>1,&\text{ for } nq^2\to \infty.\\
\end{cases}
\end{equation}
(We assume, that $\omega(n)$ tends slowly to infinity and $c$ is close to $5$ and $1$, respectively.)\\
Then independently with probability $q$ we add to the edge set each edge $(x_1,x_{2})\in\X{1}^*\times\X{2}^*$, where, for each $1\le i\le 2$, $\X{i}^*$ is the set of vertices form $\X{i}$ connected by an edge with $x$. Let $T^*(x)$ be a subgraph of $T(x)$ induced on $\X{1}\cup\X{2}$. By definition the following statements are equivalent:
\begin{equation}\label{RownanieCouplingGHGT}
%H^*(x)
\Hk{2}{q^3}
\coup{o(1/n)} T^*(x)
\end{equation}
\begin{equation}\label{RownanieGHxTx}
G H(x)\coup{o(1/n)} T(x). 
\end{equation}
Moreover
\begin{equation}\label{RownanieCouplingT*Hk-1}
 \bigcup_{x\in\X{3}}T^*(x)\coup{o(1)}\Hk{2}{a_n(q)q},
\end{equation}
where $\Hk{2}{a_n(q)q}$ is independent of the choice of $\X{i}^*$,
implies
\begin{equation*}\label{RownanieCouplingSumaGTGq}
\bigcup_{x\in\X{3}}T(x)\coup{o(1)}\Gkn{3}{a_n(q)q}. 
\end{equation*}
Therefore by \eqref{RownanieH3sumaHx} we have that \eqref{RownanieCouplingGHGT}
 and \eqref{RownanieCouplingT*Hk-1} imply the thesis. 

First we concentrate on showing \eqref{RownanieCouplingGHGT}. The proof varies for $q$ in different ranges, therefore it is divided into 4 cases:\\
CASE 1: $q=O(\ln n/n)$,\\
CASE 2: $\ln n/n=o(q)$ and $q=O(n^{-2/3}\ln^{1/3} n)$,\\
CASE 3: $n^{-2/3}\ln^{1/3} n=o(q)$ and $q=o(n^{-1/2})$,\\
CASE 4: $q=\Omega(n^{-1/2})$ and $q=o(n^{-1/3})$.

\czesc{CASE 1}
For $q=O(\ln n/n)$ with probability $1+o(1/n)$ a graph $\Hk{2}{q^3}$ consists of at most one edge. 
Namely probability that $\Hk{2}{q^3}$ has more than one edge is at most 
$$
\binom{n^2}{2}q^6=O\left(n^4q^6\right)=o\left(\frac{1}{n}\right)
$$ 
Moreover, for large $n$,
\begin{align*}
\Pr\left\{\exists_{x_1\in \X{1},x_2\in \X{2}}(x_1,x_2)\in E(T^*(x))\right\}
&\ge \sum_{x_1\in \X{1},x_2\in \X{2}}\Pr\left\{(x_1,x_2)\in E(T^*(x))\right\}\\
&-\sum_{x_1,x_1'\in \X{1},x_2,x_2'\in \X{2}}\Pr\left\{(x_1,x_2),(x_1',x_2')\in E(T^*(x))\right\}\\
&=n^2(Cq)^2q-\binom{n}{2}^2(Cq)^4q^2-2n\binom{n}{2}(Cq)^3q^2=\\
&=C^2n^2q^3(1-O(n^2q^3+nq^2))\ge n^2q^3 \ge\\
&\ge \Pr \left\{\exists_{x_1\in \X{1},x_2\in \X{2}}(x_1,x_2)\in E(\Hk{2}{q^3})\right\}.
\end{align*}
This gives an obvious coupling 
$$
\Hk{2}{q^3}\coup{o(1/n)}T^*(x).
$$

\czesc{CASES 2, 3 and 4}

If $\ln n/n = o(q)$ then the number of vertices in $\X{i}^*$ is sharply concentrated around its expected value. Let $\Hnq{2}{q}$ be a graph constructed by a similar procedure as $T^*(x)$ but with $\X{i}^*$ replaced by $\X{i}'$ chosen uniformly at random from all subsets of cardinality sufficiently smaller than $\E |\X{i}^*|$. Namely in $\Hnq{2}{r}$ first 
$\X{i}'$ is chosen uniformly at random from all $C'nq$ element subsets of $\X{i}$, where
$$
C'=\begin{cases}
   5,&\text{ for }nq^{2}=o(1);\\ 
   \omega'(n),&\text{ for }nq^{2}=\Theta(1);\\ 
   c' nq^{2}, \text{ where }1<c<c'&\text{ for }nq^{2}\to \infty,
   \end{cases}
$$
 and then each edge $(x_1,x_2)\in \X{1}'\times \X{2}'$ is added to the edge set of $\Hnq{2}{r}$ independently with probability $r$, $r\in [0;1]$. By Chernoff's bound \eqref{RownanieChernoff}
$$
\Hnq{2}{q}\coup{o(1/n)} T^*(x).
$$
Therefore
$$
\Hk{2}{q^{3}}\coup{o(1/n)}\Hnq{2}{q^{3}}
$$
implies \eqref{RownanieCouplingGHGT}.

\czesc{CASE 2}
If $q=O(n^{-2/3}\ln^{1/3} n)$ then $\Hk{2}{q^{3}}$ {\whp} does not contain many edges except a maximum matching. Therefore a coupling is constructed by comparison of the sizes of maximum matchings in $\Hk{2}{q^{3}}$ and $\Hnq{2}{q}$.

\begin{lem}\label{LematMatchingSize}
Let $r=o(1/(C'nq))$, $\ln n=o(nq)$ and $N_2(r)$ be a random variable denoting the size of a maximum matching in $\Hnq{2}{r}$, then 
\begin{equation}\label{RownanieNt}
N'_2(r)\coup{o(1/n)} N_2(r),
\end{equation}
where $N'_2(r)$ has the binomial distribution $\Bin(C'nq,s_2(r))$ and
\begin{align*}
s_2(r)=
1-\exp(-(C'nq-\sqrt{3C'nq\ln n})(1-(1-r)^{C'nq})/C'nq)\sim C'nq r.
\end{align*}
\end{lem}
\begin{proof}[Proof of Lemma~\ref{LematMatchingSize}]
Let $H$ be a hypergraph chosen according to the probability distribution of $\Hnq{2}{r}$. Define $H'$ to be a hypergraph with vertex set 
$\X{1}$ 
and edge set 
$\{
(x_1)
 : 
 x_1\in\X{1}
 \text{ and }
 \exists_{x_{2}\in\X{2}}
(x_1,x_2)
\in 
E(H)
\}.$
Notice that $H'$ is chosen according to the probability distribution  of $\Hnq{1}{1-(1-r)^{C'nq}}$ (in analogy to $\Hnq{2}{\cdot}$, $\Hnq{1}{1-(1-r)^{C'nq}}$ is a hypergraph with vertex set $\X{1}$ and edge set constructed by first choosing $\X{1}'$ uniformly at random from all $C'nq$--element subsets of $\X{1}$ and then adding to an edge set each $x_1\in\X{1}'$ independently with probability $1-(1-r)^{C'nq}$).
Let $H''$ be a subhypergraph of $H$ such that for each edge $(x_1)\in E(H')$ we pick uniformly at random an edge from $E(H)$ containing $x_1$ and add it to the edge set of $H''$. Notice that a maximum matching in $H$ is at least of the size of the set of non isolated vertices in $\X{2}$ in $H''$. Moreover the edge set of $H''$ may be alternatively constructed in the following way (i.e. this construction leads to the same probability distribution). First we pick an integer according to the binomial distribution $\Bin(C'nq,1-(1-r)^{C'nq})$, then, given the value of the picked integer, we pick a subset $\X{1}''$ uniformly at random from all subsets of $\X{1}$ of this cardinality. Independently we choose $\X{2}'$ uniformly at random from all $C'nq$--element subsets of $\X{2}$. Then to each vertex $x_1\in\X{1}'$, to create an edge, we add one vertex, chosen uniformly at random from the set $\X{2}'$. For all $x_1\in\X{1}'$ the choices of the second vertex are independent with repetition. Therefore by the above construction, \eqref{RownanieChernoffPoisson} and Fact~\ref{FactCouplingWarunkowe}
\begin{equation*}
X(M)\coup{o(1/n)} X(C'nq) \coupling{1} N_{2},
\end{equation*}
where $X(M)$ and $X(C'nq)$ are defined as in \eqref{RownanieDefinicjaXM} for $K=2$, $n_2=C'nq$, $P_2=(1-(1-r)^{C'nq})/(C'nq)$
and $M$ with the Poisson distribution $\Po(C'nq-\sqrt{3C'nq\ln n})$. Thus by Fact~\ref{FaktXMPoisson}
$X(M)$ has the binomial distribution $\Bin(C'nq, s_{2}(r))$. 
\end{proof}

The above lemma is used to show existence of a coupling between a random variable $M_2$ denoting the size of an edge set in $\Hk{2}{q^{3}}$ and $N_2$.  

\begin{lem}\label{LematCouplingMatching}
Let $C' = 5$, $M_{2}$ has the binomial distribution $\Bin(n^{2},q^3)$ and let $N_{2}$ be the size of a maximum matching in $\Hnq{2}{q}$. Then
 $$
 M_{2}\coup{o(1/n)}N_{2}.
 $$  
\end{lem}

\begin{proof}
By previous lemma and Fact~\ref{FaktCouplingPrzechodniosc} it is sufficient to show
\begin{equation}\label{RownanieCouplingMNprim}
 M_{2}\coup{o(1/n)}N'_{2},
\end{equation}
where  $N'_{2}$ has the binomial distribution $\Bin(C'nq, s_{2}(q))$ and $s_{2}(q)\sim C'nq^{2}$.\\
Notice that
\begin{align*}
M_{2}&
=\sum_{i=1}^{nq}\xi_i, 
\text{ where $\xi_i$ are independent with distribution }
\Bin\left(\frac{n}{q},q^3\right);\\
N_{2}'&
=\sum_{i=1}^{nq}\zeta_i, 
\text{ where $\zeta_i$ are independent with distribution }
\Bin\left(C',s_{2}(q)\right).
\end{align*}
Since
$s_{2}(q)\sim C'nq^{2}$,
for large $n$ we have
\begin{align*}
 \forall_{1\le l\le 4}\Pr\{\xi_i=l\}
&\le\frac{1}{l!}(nq^{2})^l\le\frac{(C')_l}{l!}s_{2}^{l}(1-s_{2})^{C'-l}
=\Pr\{\zeta_i=l\}\\
\intertext{and}\\
\Pr\{\xi_i>4\}&\le \binom{\frac{n}{q}}{5}q^{5\cdot 3}\le \left(nq^{2}\right)^5
=
\frac{1}{n^2q}\left(nq^{\frac{3}{2}}\right)^{7}q^{\frac{1}{2}}=o\left(\frac{1}{n^2q}\right)
\end{align*}
Therefore, for all $1\le i\le nq$ it is simple to construct a probability measure on $\mathbb{N}\times\mathbb{N}$, the existence of which implies
$$
\xi_i\coup{o(1/n^2p)}\zeta_i.
$$
This by Fact~\ref{FaktCouplingSuma}
implies \eqref{RownanieCouplingMNprim}.
\end{proof}

\noindent Let $\mathcal{G}$ be a set of $2$--partite graphs with $2$--partition $(\X{1},\X{2})$. We define\\ 
\begin{tabular}{llp{14cm}}
$\M(l)$&-&the subset of $\mathcal{G}$ containing all graphs with a maximum matching of cardinality~$l$;\\
$\M_1(l)$&-&the subset of $\M(l)$ containing all graphs with the maximum degree $1$;\\
$\M_2(l)$&-&the subset of $\M(l)$ containing all graphs with the maximum degree $2$ and exactly one vertex of degree $2$\\
\end{tabular}\\
and 
\begin{equation*}
\M_1=\bigcup_{l=0}^n\M_1(l),\quad \M_2=\bigcup_{l=0}^n\M_2(l).
\end{equation*}
For $q=o(n^{-2/3})$
\begin{align}\label{RownanieTylkoMatching}
\Pra{\Hk{2}{q^3}\notin \M_1}
\le
2n\binom{n}{2}q^{6}=O\left(n^{4}q^{6}n^{-1}\right)=o\left(\frac{1}{n}\right) 
\end{align}
and for $q=O\left((n^{-2}\ln n)^{1/3}\right)$
\begin{equation}\label{RownanieTylkoMatching2}
\begin{split}
\Pr&\{\Hk{2}{q^3}\notin \M_1\cup\M_2\}\le\\
&\le
2\binom{n}{2}\left(\binom{n}{2}q^{6}\right)^2
+n^2\left(\binom{n-1}{2}q^{6}\right)^2
+n^2q^3\left(\left(n-1\right)q^3\right)^2
+2n\binom{n}{3}q^{9}=\\
&=O
\left(
n^{6}q^{12}
+n^{4}q^{9}
\right)=
O\left(
\left(n^{2}q^3\right)^4n^{-2}+
\left(n^{2}q^3\right)^3n^{-2}\right)
=
o\left(\frac{1}{n}\right)
\end{split}
\end{equation}
Now let
$$
\mu:\mathbb{N}\times\mathbb{N}\to [0,1]
$$
be a probability measure associated with a coupling of $M_{2}$ and $N_{2}$ existing by Lemma~\ref{LematCouplingMatching}. Starting with the probability measure $\mu$ we construct a coupling, which implies for large $n$
$$
\Hk{2}{q^3}
\coup{o(1/n)}\Hnq{2}{q}.
$$

Let $H^{(2)}(\M_1)$ be a random graph constructed by first sampling $H$ according to the probability distribution of $\Hk{2}{q^3}$ and replacing it by a graph chosen uniformly at random from $\M_1(|E(H)|)$ in the case where $H\notin\M_1$. Moreover let $H^{(2)}(\M_1\cup\M_2)$ be a random graph constructed by sampling $H$ according to the probability distribution of $\Hk{2}{q^3}$  and replacing it by a graph chosen uniformly at random from $\M_1(|E(H)|)$ in the case where $H\notin \M_1\cup\M_2$. Sizes of edge sets of $H^{(2)}(\M_1)$ and $H^{(2)}(\M_1\cup\M_2)$ are random variables  $M_{2}(\M_1)$ and $M_{2}(\M_1\cup\M_2)$, respectively. Obviously $M_{2}(\M_1)$ and $M_{2}(\M_1\cup\M_2)$ have the same distribution as $M_{2}$. 
For any event $\mathcal{A}$, denote by $H^{(2)}(\M_1)^{[\mathcal{A}]}$, $H^{(2)}(\M_1\cup\M_2)^{[\mathcal{A}]}$ and $\Hnq{2}{q}^{[\mathcal{A}]}$ graphs $H^{(2)}(\M_1)$, $H^{(2)}(\M_1\cup\M_2)$ and $\Hnq{2}{q}$ under condition~$\mathcal{A}$.

Let $q=o(n^{-2/3})$.
By \eqref{RownanieTylkoMatching}
$$
\Hk{2}{q^3}\coup{o(1/n)}H^{(2)}(\M_1),
$$ 
Therefore it remains to show
$$
H^{(2)}(\M_1)\coup{o(1/n)}\Hnq{2}{q}.
$$ 
Let $(l_1,l_2)\in \mathbb{N}\times\mathbb{N}$ be chosen according to the probability measure $\mu$. If $l_1>l_2$, then we sample $H^{(2)}(\M_1)^{[M_{2}=l_1]}$ and $\Hnq{2}{q}^{[N_{2}=l_2]}$ independently. And if $l_1\le l_2$, then first we sample an instance of $\Hnq{2}{q}^{[N_{2}=l_2]}$ and then choose its subgraph uniformly at random from all its subgraphs contained in $\M_1(l_2)$. Then, from the chosen subgraph, we delete $l_2-l_1$ edges chosen uniformly at random. Thereby we get the edge set of $H^{(2)}(\M_1)^{[M_{2}=l_1]}$. 

Let now $q = \Omega(n^{-2/3})$ and $q=O\left(n^{-2/3}\ln^{1/3} n\right)$. Let also
\begin{align*}
P_1(l)&=\Pr\{H^{(2)}(\M_1\cup\M_2)^{[M_2=l]}\in\M_1\};
&
P_2(l)&=\Pr\{\Hnq{2}{q}^{[N_{2}=l]}\in\M_1\};\\ 
Q_1(l)&=\Pr\{H^{(2)}(\M_1\cup\M_2)^{[M_2=l]}\notin\M_1\};
&
Q_2(l)&=\Pr\{\Hnq{2}{q}^{[N_{2}=l]}\notin\M_1\}.\\ 
\end{align*}
By \eqref{RownanieTylkoMatching2} we are left with showing that
$$
H^{(2)}(\M_1\cup\M_2)\coup{o(1/n)}\Hnq{2}{q}.
$$ 
Let $(l_1,l_2)\in \mathbb{N}\times\mathbb{N}$ be chosen according to the probability measure $\mu$.  If $l_1>l_2$ or $l_2\ge \omega(n)\ln n$ (where $\omega(n)$ is a sequence tending slowly to infinity), then we construct a pair of graphs from $\mathcal{G}$ by sampling independently $H^{(2)}(\M_1\cup \M_2)^{[M_2=l_1]}$ and $\Hnq{2}{q}^{[N_{2}=l_2]}$. If $1\le l_1\le l_2 < \omega(n)\ln n$, then we sample $H$, a second graph in a pair,  according to the probability distribution of $\Hnq{2}{q}^{[N_{2}=l_2]}$. If $H\in \M_1$, then we choose  a first graph uniformly at random from all subgraphs of $H$ contained in $\M_1(l_1)$. If $H\notin \M_1$, then with probability $(P_1(l_1)-P_2(l_2))/Q_2(l_2)$ we choose a first graph
uniformly at random from all subgraphs of $H$ contained in $\M_1(l_1)$ and  with probability $Q_1(l_1)/Q_2(l_2)$  we choose a first
graph
 uniformly at random from all subgraphs of $H$ contained in $\M_2(l_1)$. 

According to this construction the first graph is chosen according to the probability distribution of $H^{(2)}(\M_1\cup\M_2)$ and the second according to the probability distribution of $\Hnq{2}{q}$. Moreover
$$
\mu(\{(l_1,l_2):l_1>l_2\})=o\left(\frac{1}{n}\right).
$$
In addition, the size of a maximum matching (i.e. $N_{2}$) is at most the number of edges of $\Hnq{2}{q}$, which has the binomial distribution with expected value $(C'n)^{2}q^3=O(\ln n)$. Thus  by Chernoff's bound~\eqref{RownanieChernoff}
$$
\mu(\{(l_1,l_2):l_2\ge \omega(n)\ln n\})=o\left(\frac{1}{n}\right)
$$
Therefore this is a desired coupling and it is well defined for large $n$ if $P_1(l_1)\ge P_2(l_2)$ for large $n$ and $l_1\le l_2$.  
Calculations show that for a given $l<\omega(n)\ln n$ and $\omega(n)$ tending slowly to infinity
\begin{align*}
Q_1(l)&\le
1-\frac{\binom{n}{l}^{2}(l!)}{\binom{n^{2}}{l}}
=
1-\prod_{i=0}^{l-1}\left(\frac{(n-i)^{2}}{n^{2}-i}\right)
=\\
&=
1-\prod_{i=0}^{l-1}\left(1-\frac{2ni-i^2}{n^{2}-i}\right)
\le
1-\prod_{i=0}^{l-1}\left(1-\frac{2l}{n}\right)\le \frac{2l^2}{n}
\intertext{and}
Q_2(l)&=\Pr\{\Hnq{2}{q}^{[N_{2}=l]}\notin\M_1\}\ge\\
&\ge \frac{\Pr\{\Hnq{2}{q}^{[N_{2}=l]}\notin\M_1\}-\Pr\{\Hnq{2}{q}^{[N_{2}=l]}\notin\M_1\cup\M_2\}}{1-\Pr\{\Hnq{2}{q}^{[N_{2}=l]}\notin\M_1\cup\M_2\}}=\\
&=
\frac{\Pr\{\Hnq{2}{q}^{[N_{2}=l]}\in\M_2\}}{\Pr\{\Hnq{2}{q}^{[N_{2}=l]}\in\M_1\}+\Pr\{\Hnq{2}{q}^{[N_{2}=l]}\in\M_2\}}=\\
&=
\frac{\Pr\{\Hnq{2}{q}\in\M_2(l)\}}{\Pr\{\Hnq{2}{q}\in\M_1(l)\}+\Pr\{\Hnq{2}{q}\in\M_2(l)\}}=\\
&=\Omega\left(nq^{2}l\right)=\Omega(n^{-1/3}l),
\end{align*}
since
\begin{align*}
\Pr&\{\Hnq{2}{q}\in\M_2(l)\}=\\
&=\binom{C'nq}{l+1}\binom{C'nq}{l}\binom{l+1}{2}(l!)
\left(\frac{q}{1-q}\right)^{l+1}
\left(1-q\right)^{(C'nq)^{2}}
=\\
&=
\binom{C'nq}{l}^{2}\frac{(C'nq-l)}{(l+1)}
\cdot 
\frac{(l+1)l}{2}
(l!)
\left(\frac{q}{1-q}\right)^{l}
\frac{q}{1-q}
\left(1-q\right)^{(C'nq)^{2}}
=\\
&=\Pr\{\Hnq{2}{q}\in\M_1(l)\}
(1+o(1))\frac{C'nq^2l}{2}
\end{align*}

Hence $Q_1(l_1)=o(Q_2(l_2))$ uniformly over all $1\le l_1\le l_2\le \omega(n)\ln n$ and $\omega(n)$ such that $\omega(n)\ln n=o(nq)$.

\czesc{CASE 3 and 4}

If $n^{-2/3}\ln^{1/3} n\ll q$ the numbers of edges in $\Hk{2}{q^3}$ and $\Hnq{2}{q}$ are sharply concentrated around their expected values. 

Let $H^{**}(x)$ and $T^{**}(x)$ be random bipartite multigraphs with $2$-partition $(\X{1},\X{2})$ and $(\X{1}',\X{2}')$, respectively, with the numbers of edges with the Poisson distribution $\Po(-n^{2}\ln (1-q^3))$ and $\Po(-(C'nq)^{2}\ln (1-q))$, respectively, and an edge sets constructed by independently choosing one by one with repetition edges from 
$\{(x_1,x_2):x_1\in \X{1}\text{ and }x_2\in \X{2}\}$ and 
$\{(x_1,x_2):x_1\in \X{1}'\text{ and }x_2\in \X{2}'\}$, respectively. By Fact~\ref{FaktXMPoisson} $\Hk{2}{q^3}$ and $\Hnq{2}{q}$ are underlying graphs of $H^{**}(x)$ and $T^{**}(x)$.

Let moreover $H^{***}(x)$ and $T^{***}(x)$ be multigraphs constructed in an analogous manner but with $C_1n^{2}q^3$ ($C_1>1$ is a constant) and $(C'')^{2}n^{2}q^3$ edges (where $C'/C'' > 1$ are constants), respectively. 
By Chernoff's bound \eqref{RownanieChernoffPoisson} 
$$H^{**}(x)\coup{o(1/n)}H^{***}(x)\quad\text{ and }\quad T^{***}(x)\coup{o(1/n)}T^{**}(x).$$ 

Notice that choosing an edge in above defined multigraphs is equivalent to choosing its $2$ vertices independently from each set of $2$--partition. Therefore, instead of choosing each edge one by one, first a degree sequence in each set $\X{i}$ ($\X{i}'$) may be chosen and on this basis a multigraph with a given degree sequence may be crated. 
Let $\D{1}{j}$ be the random variable denoting the degree of the $j$-th vertex in $\X{i}$ in $H^{***}(x)$ and $\D{5}{j}$ be the random variable denoting the degree of the $j$-th vertex in $\X{i}$ in $T^{***}(x)$. By Fact~\ref{FaktCouplingSuma}
$$
(\D{1}{1},\ldots,\D{1}{n})\coup{o(1/n)}(\D{5}{1},\ldots,\D{5}{n}),\text{for each $\X{i}$, $1\le i\le 2$}
$$
 imply
$$
H^{***}(x)\coup{o(1/n)} T^{***}(x),
$$

We introduce auxiliary urn models. Assume that we have $n$ urns. Let $\D{i}{}=(\D{i}{1},\ldots,\D{i}{n})$ be the random vector in which $\D{i}{j}$ represents the number of balls in the $j$-th urn in the $i$--th model. Let $1<C_1<C_2$, $C'''/C''''>1$, $C''/C'''>1$ and $C'/C''>1$ be constants such that $C_2<C''''$. 
\begin{itemize}
\item In the $1$--st model we throw $C_1 n^{2}q^3$ balls one by one  independently, with repetition, to the urn chosen uniformly at random from $n$ urns.
\item In the $2$--nd model the number of thrown balls has the Poisson distribution $\Po(C_2n^{2}q^3)$, i.e. by Fact~\ref{FaktXMPoisson}   $\D{2}{j}$ has the Poisson distribution $\Po(C_2nq^3)$ (for $K=2$, $P_2=\frac{1}{n}$, $n_2=n$).
\item In the $3$--rd model $\D{3}{j}=D_j\cdot D_j'$, where $D_j$ is a Bernoulli random variable with probability of success $C''''q$ and $D_j'$ has the Poisson distribution $\Po(C''''nq^{2})$.
\item In the $4$--th model first we select $C'''nq$ urns from the set of all urns and the number of  balls thrown to the selected urns has the Poisson distribution $\Po((C''')^{2}n^{2}q^3)$, i.e. for the urns not selected  $\D{4}{j}=0$ and for the selected urns $\D{4}{j}$ has the Poisson distribution $\Po(C'''nq^{2})$ (by Fact~\ref{FaktXMPoisson} for $K=2$, $P_2=\frac{1}{C'''nq}$, $n_2=C'''nq$). 
\item In the $5$--th model first we select $C''nq$ urns and we throw $(C'')^{2}n^{2}q^3$ balls one by one  independently to the urn chosen uniformly at random from the set of selected urns.
\end{itemize}
By Chernoff's bound
$$
\D{1}{}\coup{o(1/n)}\D{2}{}\quad \text{and}\quad\D{3}{}\coup{o(1/n)}\D{4}{}\coup{o(1/n)}\D{5}{}.
$$
Moreover, by Fact~\ref{FaktCouplingSuma}, if for large $n$
\begin{equation}\label{RownanieCouplingDj}
\D{2}{j}\coup{o(1/n^2)}D_j\cdot D_j',
\end{equation}
then for large $n$
\begin{equation}\label{RownanieCouplingD}
\D{2}{}\coup{o(1/n)} \D{3}{}.
\end{equation}
The constants may be chosen such that
$$
C''''=\begin{cases}
   4,&\text{ for }nq^{2}=o(1);\\ 
   \omega''''(n),&\text{ for }nq^{2}=\Theta(1);\\ 
   c'''' nq^{2},&\text{ for }nq^{2}\to \infty,
   \end{cases}
$$
where $c''''>1$ and $\omega''''(n)$ is a function tending slowly to infinity .

\noindent For large $n$
\begin{align*}
\Pra{\D{2}{j}\ge 1}
&= 1-\exp(-C_2nq^3)
\le\\
\label{RownanieStopien1}
&\le C''''
q
\left(
1-
\exp\left(-(C'''')nq^{2}\right)\right)=\\
&=\Pra{\D{3}{j}\ge 1}.
\end{align*}

\noindent Moreover, for $t\ge 2$, $nq^2=o(1)$ and large $n$ 
\begin{align*}
\Pra{\D{2}{j}\ge t}&\sim \frac{(C_2nq^3)^t}{t!} = o\left(  C''''q\frac{(C'''nq^{2})^t}{t!}\right)=o\left(\Pra{\D{3}{j}\ge t}\right)
\end{align*}
This implies \eqref{RownanieCouplingDj} for $nq^2=o(1)$.
 
Let now $nq^2=\Omega(1)$. By Chernoff's bound, if we estimate the number of urns with at least one ball and compare it to the number of balls we get, with probability $1-o(1/n)$ the number of urns with at least $2$ balls in the $3$-rd model is $o(n^{2}q^3)$ and $\Omega(n^{2}q^3)$ in the $2$-nd model. Therefore, since urns with at least $2$ balls are uniformly distributed, a coupling is easy to construct.

This completes the proof of \eqref{RownanieCouplingGHGT}. It remains to prove \eqref{RownanieCouplingT*Hk-1}.

\czesc{Proof of \eqref{RownanieCouplingT*Hk-1}}

Let $C=C(q)$ be defined as in \eqref{RownanieDefinicjaC}.
Define  
$X_n(x_1,x_{2})=|\{x\in\X{3}: x_1\in \X{1}^*(x), \ldots, x_{2}\in \X{2}^*(x)\}|$. 
$X_n(x_1,x_{2})$ has the binomial distribution $\Bin(n,(Cq)^{2})$ and 
$$
\E X_n=C^2nq^2=
\begin{cases}
cnq^2,\text{ where } c>25&\text{ for } nq^2=o(1);\\
\omega^2(n)nq^2,&\text{ for } nq^2=\Theta(1);\\
cn^3q^{6},\text{ where } c>1,&\text{ for } nq^2\to \infty.\\
\end{cases}
$$
Therefore by Lemma~\ref{RownanieChernoffDokladny1} 
\begin{equation*}
\Pr\{\exists_{(x_1,x_2)}X_n(x_1,x_2)\ge a'_n(q)\}\le n^2\Pra{X_n(x_1,x_2)\ge a'_n(q)}=o(1),
\end{equation*}
where
$$
a'_n(q)=
\begin{cases}
3\ln n/(\ln \ln n - \ln (nq^2)),&\text{ for } nq^2=o(1);\\ 
3\ln n/\ln \ln n,&\text{ for } nq^2=\Theta(1);\\
3\ln n/(\ln \ln n - \ln (n^3q^{6}),&\text{ for } nq^2\to \infty\text{ and } nq^2=o(\sqrt[3]{\ln n});\\
\omega_1(n)\: n^3q^{6},\text{ where }\omega_1(n)\to\infty&\text{ for } nq^2\to \infty\text{ and } nq^2=\Theta(\sqrt[3]{\ln n});\\
cn^3q^{6},\text{ where }c>1&\text{ for } nq^2\to \infty\text{ and } o(nq^2)=\sqrt[3]{\ln n},\\
\end{cases}
$$
(since $\omega(n)$ tends to infinity arbitrarily slowly).\\
By definition, probability that there is an edge connecting $x_1$ and $x_2$ in $\bigcup_{x\in\X{3}}T^*(x)$ is at most $X_n(x_1,x_{2})\cdot q$, thus
$$
\bigcup_{x\in\X{3}}T^*(x)\coup{o(1)}\Hk{2}{a_n(q)q}.
$$
\end{proof}

\begin{proof}[Proof of Lemma~\ref{TwierdzenieCoupling4} and \ref{TwierdzenieCoupling5}]

Let $k=4$ or $k=5$. For $x\in\X{k}$, let $H(x)$ be a hypergraph with vertex set $\{x\}\cup\bigcup_{i=1}^{k-1}\X{i}$ and edge set consisting of those edges from $E\left(\Hk{k}{q^{\binom{k}{2}}}\right)$, which contain $x$. Then  
\begin{equation}\label{RownanieSumaHx}
\Hk{k}{q^{\binom{k}{2}}}=\bigcup_{x\in\X{k}}H(x).
\end{equation}

Let $T(x)$ be an auxiliary hypergraph,
with vertex set $\{x\}\cup\bigcup_{i=1}^{k-1}\X{i}$ and edge set constructed by the following procedure. First we add each edge $(x,y)$, $y\in\bigcup_{i=1}^{k-1}\X{i}$ independently with probability $C q$ ($C>5$) to the edge set and then independently with probability $q^{\binom{k-1}{2}}$ we add to the edge set each edge $(x_1,\ldots,x_{k-1})\in\X{1}^*\times\ldots\times\X{k-1}^*$, where, for each $1\le i\le k-1$, $\X{i}^*$ is the set of vertices connected by an edge with $x$. Let moreover $T^*(x)$ be a subhypergraph of $T(x)$ induced on $\bigcup_{i=1}^{k-1}\X{i}$. 
Recall that \eqref{RownanieCouplingGHGT} implies \eqref{RownanieGHxTx}. Similarly
\begin{equation}\label{2RownanieCouplingGHGT}
\Hk{k-1}{q^{\binom{k}{2}}}\coup{o(1/n)} T^*(x)
\end{equation}
implies
\begin{equation}\label{2RownanieCouplingGHxTx}
G H(x)\coup{o(1/n)} G T(x). 
\end{equation}

\noindent Moreover if for some constant $c>5(k-1)$
\begin{equation}\label{2RownanieCouplingT*Hk-1}
 \bigcup_{x\in\X{k}}T^*(x)\coup{o(1)}\Hk{k-1}{(cq)^{\binom{k-1}{2}}},
\end{equation}
where $\Hk{k-1}{(cq)^{\binom{k-1}{2}}}$ is independent of choices of $\X{i}^*$
and
\begin{equation}\label{RownanieLematk-1}
\Hk{k-1}{(cq)^{\binom{k-1}{2}}}\coup{o(1)}\Gkn{k-1}{a_n(c'_kq)c'_kq}
\end{equation}
then
\begin{equation*}
\bigcup_{x\in\X{k}}G T^*(x)\coup{o(1)}\Gkn{k-1}{a_n(c'_kq)c'_kq}. 
\end{equation*}
Thus also
\begin{equation}\label{2RownanieCouplingSumaGTGq}
\bigcup_{x\in\X{k}}GT(x)\coup{o(1)}\Gkn{k}{a(c'_kq)c'_kq}.
\end{equation}

Since by \eqref{RownanieSumaHx} we have that \eqref{2RownanieCouplingGHxTx} and \eqref{2RownanieCouplingSumaGTGq} imply the thesis, we are left with showing \eqref{2RownanieCouplingGHGT}, \eqref{2RownanieCouplingT*Hk-1} and \eqref{RownanieLematk-1}.
Since \eqref{RownanieLematk-1} follows by Lemma \ref{TwierdzenieCoupling3} or \ref{TwierdzenieCoupling4} for $k=4$ or $k=5$, respectively, we only need to prove \eqref{2RownanieCouplingGHGT}  and \eqref{2RownanieCouplingT*Hk-1}. 

\czesc{Proof of \eqref{2RownanieCouplingGHGT}}

The proof of \eqref{2RownanieCouplingGHGT} is similar to this of \eqref{RownanieCouplingGHGT} thus we omit many details which are the same as in the proof of \eqref{RownanieCouplingGHGT}. Under the assumptions of the lemmas $\ln n = o(nq)$. Thus in analogy to the proof of Lemma~\ref{TwierdzenieCoupling3} (in CASE 2, 3 and 4) 
$$
\Hnq{k-1}{q^{\binom{k-1}{2}}}\coup{o(1/n)} T^*(x),
$$
where $5=C'<C$ and $\Hnq{k-1}{r}$ is an analogue of $\Hnq{2}{r}$ (i.e. is created by first choosing  $\X{i}'$ uniformly at random from $C'nq$--element subsets of $\X{i}$, for all $1\le i\le k-1$, and then adding each edge $(x_1,\ldots, x_{k-1})\in \X{1}'\times\ldots\times\X{k-1}'$ independently with probability~$r$.
Therefore it remains to show that
$$
\Hk{k-1}{q^{\binom{k}{2}}}\coup{o(1/n)}\Hnq{2}{q^{\binom{k-1}{2}}}.
$$

\noindent As before the proof differ for $q$ in different ranges, therefore we will consider two cases.
\begin{itemize}
\item $q=O(n^{-2/k}\ln ^{1/\binom{k}{2}})$ (similar to CASE 2 in the proof of Lemma~\ref{TwierdzenieCoupling3})
\item $q\gg n^{-2/k}\ln ^{1/\binom{k}{2}}$ and $q=o(n^{-2/5})$ (similar to CASE 3 in the proof of Lemma~\ref{TwierdzenieCoupling3})
\end{itemize}

Let $q=O(n^{-2/k}\ln ^{1/\binom{k}{2}})$.
The lemma below is a generalisation of Lemma~\ref{LematMatchingSize} and follow by induction. The proof of an inductive step is similar to the proof of Lemma~\ref{LematMatchingSize} with slight changes.

\begin{lem}\label{AppLematMatchingSize}
Let $k\ge 3$,  $r=o(1/(C'nq)^{k-2})$, $\ln n=o(nq)$ and $N_{k-1}(r)$ be the random variable denoting the size of a maximum matching in $\Hnq{k-1}{r}$, then 
\begin{equation}\label{AppRownanieNt}
N'_{k-1}(r)\coup{o(1/n)} N_{k-1}(r),
\end{equation}
where $N'_{k-1}(r)$ has the binomial distribution $\Bin(C'nq,s_{k-1}(r))$ and
\begin{multline*}
s_{k-1}(r)=\\
=
\begin{cases}
1-\exp(-(1-(1-r)^{C'nq})(1-\sqrt{3C'nq\ln n}/(C'nq)))\sim C'nq r
&\text{ for }k=3\\
1-\exp(-s_{k-2}((1-(1-r)^{C'nq}))(1-\sqrt{3C'nq\ln n}/(C'nq)))\sim (C'nq)^{k-2} r
&\text{ for }k\ge 4
\end{cases}
\end{multline*}
\end{lem}
\begin{proof}
The proof follow by induction on $k$. By Lemma~\ref{LematMatchingSize} it remains to show an inductive step. 
Let $k\ge 4$.
Let $H$ be a hypergraph chosen according to the probability distribution of $\Hnq{k-1}{r}$. Define $H'$ to be a hypergraph with vertex set 
$\X{1}\cup\ldots\cup \X{k-2}$ 
and edge set 
$\{
(x_1,\ldots,x_{k-2})
 : 
 x_i\in\X{i}
 \text{ and }
 \exists_{x_{2}\in\X{2}}
(x_1,\ldots,x_{k-1})
\in 
E(H)
\}.$
Notice that $H'$ is chosen according to the probability distribution  of $\Hnq{k-2}{1-(1-r)^{C'nq}}$. Now let $H'_{M}$ be its subgraph with edge set chosen uniformly at random from all maximum matchings of $H'$.
Let $H''$ be a subhypergraph of $H$ such that for each edge $(x_1,\ldots,x_{k-2})\in E(H'_{M})$ we pick uniformly at random an edge from $E(H)$ containing $(x_1,\ldots,x_{k-2})$ and add it to the edge set of $H''$. A maximum matching in $H$ is at least of the size of the set of non isolated vertices in $\X{k-1}$ in $H''$. The edge set of $H''$ may be alternatively constructed in the following way. First we pick an integer according to the distribution of $N_{k-2}((1-(1-r)^{C'nq})$, then, given the value of the picked integer, we pick a matching  uniformly at random from all matchings of this cardinality with edges from $\X{1}\times\ldots\times\X{k-2}$. Independently we choose $\X{k-1}'$ uniformly at random from all $C'nq$--element subsets of $\X{k-1}$. Then to each edge from the chosen matching, in order to create an edge of $H''$, we add one vertex, chosen uniformly at random from the set $\X{k-1}'$. For all edges the choices of an additional vertex are independent with repetition. 
By Fact~\ref{FactCouplingWarunkowe}, the above construction and inductive assumption (i.e. $N'_{k-2}(1-(1-r)^{C'nq})\coup{o(1/n)}N_{k-2}(1-(1-r)^{C'nq})$) we have
\begin{equation*}
X(N'_{k-2}(1-(1-r)^{C'nq}))
\coup{o(1/n)} X(N_{k-2}(1-(1-r)^{C'nq}))
\coupling{1} N_{k-1}(r),
\end{equation*}
where $X(\cdot)$ is defined as in \eqref{RownanieDefinicjaXM} for $K=2$, $n_2=C'nq$, $P_2=1/(C'nq)$.
Moreover $X(N'_{k-2})$ for $K=2$, $n_2=C'nq$, $P_2=1/(C'nq)$ has the same distribution as  
$X(C'nq)$ for $K=2$, $n_2=C'nq$, $P_2=s_{k-2}(1-(1-r)^{C'nq})/(C'nq)$.

Therefore by \eqref{RownanieChernoffPoisson} and Fact~\ref{FactCouplingWarunkowe} 
\begin{equation*}
X(M)
\coup{o(1/n)} X(C'nq)
\coup{o(1/n)} N_{k-1},
\end{equation*}
where $X(\cdot)$ is defined for $K=2$, $n_2=C'nq$, $P_2=s_{k-2}(1-(1-r)^{C'nq})/(C'nq)$
and $M$ has the Poisson distribution $\Po(C'nq-\sqrt{3C'nq\ln n})$. Thus by Fact~\ref{FaktXMPoisson}
$X(M)$ has the binomial distribution $\Bin(C'nq, s_{k-1}(r))$. 
\end{proof}

Let $M_{k-1}$ be a random variable denoting the size of the edge set in $\Hk{k-1}{q^{\binom{k}{2}}}$.  

\begin{lem}\label{AppLematCouplingMatching}
Let $C' = 5$, $M_{k-1}$ has the binomial distribution $\Bin\left(n^{k-1},q^{\binom{k}{2}}\right)$ and let $N_{k-1}$ be the size of a maximum matching in $\Hnq{k-1}{q^{\binom{k-1}{2}}}$. Then
 $$
 M_{k-1}\coup{o(1/n)}N_{k-1}.
 $$  
\end{lem}

\begin{proof}
The proof is similar to the proof of Lemma~\ref{LematCouplingMatching}. For $k\ge 4$
\begin{align*}
M_{k-1}&
=\sum_{i=1}^{nq}\xi_i, 
\text{ where $\xi_i$ are independent with distribution }
\Bin\left(\frac{n}{q},q^{\binom{k}{2}}\right);\\
N_{k-1}'&
=\sum_{i=1}^{nq}\zeta_i, 
\text{ where $\zeta_i$ are independent with distribution }
\Bin\left(C',s_{k-1}\left(q^{\binom{k-1}{2}}\right)\right).
\end{align*}
A similar calculation to this from Lemma~\ref{LematCouplingMatching} shows that
\begin{align*}
 \forall_{1\le l\le 4}\Pr\{\xi_i=l\}=\Pr\{\zeta_i=l\}\quad
\text{and}\quad
\Pr\{\xi_i>4\}=o\left(\frac{1}{n^2q}\right),
\end{align*}
which imply the thesis of Lemma~\ref{AppLematCouplingMatching}
\end{proof}

Let $\mathcal{G}$ be a set of $k-1$--partite graphs with $k-1$--partition $(\X{1},\ldots,\X{k-1})$. 
Define  $\M(l)$, $\M_1(l)$, $\M_2(l)$, $\M_1$, $\M_2$ as in the proof of Lemma~\ref{TwierdzenieCoupling3}.

\noindent For $q=o(n^{-2/k})$
\begin{align}\label{AppRownanieTylkoMatching}
\Pra{\Hk{k-1}{q^{\binom{k}{2}}}\notin \M_1}
\le
(k-1)n\binom{n^{k-2}}{2}\left(q^{\binom{k}{2}}\right)^{2}
=o\left(\frac{1}{n}\right) 
\end{align}
and similarly for $q=O\left(n^{-2/k}\ln^{1/\binom{k}{2}} n\right)$
\begin{align*}
\Pr&\{\Hk{k-1}{q^{\binom{k}{2}}}\notin \M_1\cup\M_2\}=
o\left(\frac{1}{n}\right).
\end{align*}

\noindent For $q=\Omega(n^{-2/k})$, a given $l<\omega(n)\ln n$ and $\omega(n)$ tending slowly to infinity we have
\begin{align*}
&\Pr\{\Hnq{k-1}{q^{\binom{k-1}{2}}}\in\M_2(l)\}=\\
&=\Pr\left\{\Hnq{k-1}{q^{\binom{k-1}{2}}}\in\M_1(l)\right\}
(1+o(1))\frac{(C'n)^{k-1}q^{\binom{k-2}{2}}l}{2}.
\end{align*}
Therefore
\begin{align*}
Q_1(l)&\le
1-\frac{\binom{n}{l}^{k-1}(l!)^{k-1}}{\binom{n^{k-1}}{l}}
\le \frac{(k-1)l^2}{n}
\intertext{and}
Q_2(l)&\ge
\frac{\Pr\left\{\Hnq{k-1}{q^{\binom{k-1}{2}}}\in\M_2(l)\right\}}{\Pr\left\{\Hnq{k-1}{q^{\binom{k-1}{2}}}\in\M_1(l)\right\}+\Pr\left\{\Hnq{k-1}{q^{\binom{k-1}{2}}}\in\M_2(l)\right\}}=\\
&=\Omega\left(n^{k-2}q^{\binom{k}{2}-1}l\right)=\Omega\left(\frac{l^2}{n}\cdot\frac{l}{q}\right).
\end{align*}
Thus $Q_1(l)=o(Q_2(l))$ and the same couplings as those presented in the proof of CASE 2 of Lemma~\ref{TwierdzenieCoupling3} but with $\Hnq{2}{q}$ and $\Hk{2}{q^{3}}$ replaced by $\Hnq{k-1}{q^{\binom{k-1}{2}}}$ and $\Hk{k-1}{q^{\binom{k}{2}}}$ imply the thesis.

Let $n^{-2/k}\ln^{1/\binom{k}{2}} n= o(q)$. In this case the numbers of edges in $\Hk{k-1}{q^{\binom{k-1}{2}}}$ and $\Hnq{k-1}{q^{\binom{k-1}{2}}}$ are sharply concentrated around their expected values. In analogy to the proof of Lemma~\ref{TwierdzenieCoupling3} we define $H^{**}(x)$, $T^{**}(x)$,  $H^{***}(x)$, $T^{***}(x)$ and $\D{i}{j}$ for $1\le i\le 5, 1\le j\le n$. Therefore 
$\D{2}{j}$ has the Poisson distribution $\Po (C_2 n^{k-2}q^{\binom{k}{2}})$ and $\D{2}{j}=D_j\cdot D_j'$, where $D_j$ is a Bernoulli random variable with probability of success $C''''q$ and $D_j'$ has the Poisson distribution 
$\Po ((C'''' n)^{k-2}q^{\binom{k}{2}-1})$ for $C''''=4$.
Thus calculation shows that for large $n$
\begin{align*}
\Pra{\D{2}{j}\ge 1}
\le \Pra{\D{3}{j}\ge 1}.
\end{align*}
and for $t\ge 2$, $q=o(n^{-2/5})$ and large $n$ 
\begin{align*}
\Pra{\D{2}{j}\ge t}=o\left(\Pra{\D{3}{j}\ge t}\right)
\end{align*}
This implies \eqref{RownanieCouplingDj} for $k=4,5$ and proves \eqref{2RownanieCouplingGHGT} in the case: $n^{-2/k}\ln^{1/\binom{k}{2}} n=o(q)$ and $q=o(n^{-2/5})$.

\czesc{Proof of \eqref{2RownanieCouplingT*Hk-1}}

Define  
$X_n=X_n(x_1,\ldots,x_{k-1})=|\{x\in\X{k}: x_1\in \X{1}^*(x) \ldots x_{k-1}\in \X{k-1}^*(x)\}|$. 
It has the binomial distribution $\Bin(n,(Cq)^{k-1})$ and for large $n$
$$
\E X_n=C^{3}nq^{3}\le n^{-1/5} \quad\text{and}\quad \E X_n=C^{4}nq^{4}\le n^{-3/5}.
$$
Therefore, since
$$
\frac{\ln n}{\ln \ln n - \ln n^{-1/5}}\sim 5 \quad\text{and}\quad \frac{\ln n}{\ln \ln n - \ln n^{-3/5}}\sim \frac{5}{3},
$$
by Lemma~\ref{RownanieChernoffDokladny} for any constant $c_4''>15$ and $c_5''>20/3$
\begin{equation*}
\Pr\{\exists_{(x_1,\ldots,x_{k-1})}X_n(x_1,\ldots,x_{k-1})\ge c_k''\}\le n^{k-1}\Pra{ X_n(x_1,\ldots,x_{k-1})\ge c_k''}=o(1).
\end{equation*}
Thus in the case $k=4$, for any constant $c'_4>\sqrt[3]{15}$, we have
$$
\bigcup_{x\in\X{4}}T^{*}(x)\coup{o(1)}\Hk{3}{(c'_4q)^{3}}.
$$
Thus, by Lemma~\ref{TwierdzenieCoupling3},
$$
\bigcup_{x\in\X{4}}T^{*}(x)\coup{o(1)}\Gkn{3}{a_n(c'_4q)c'_4q}.
$$
This implies the thesis of Lemma~\ref{TwierdzenieCoupling4}.

Analogously, for $k=5$ and $c>\sqrt[6]{20/3}$
$$
\bigcup_{x\in\X{5}}T^{*}(x)\coup{o(1)}\Hk{4}{(cq)^{6}}.
$$
Therefore by Lemma~\ref{TwierdzenieCoupling4}, for $c'_5>\sqrt[6]{20/3}\sqrt[3]{15}=\sqrt[6]{2^2\cdot 3\cdot 5^3}$
$$
\bigcup_{x\in\X{5}}T^{*}(x)\coup{o(1)}\Gkn{4}{a_n(c'_5q)c'_5q},
$$
which implies the thesis of Lemma~\ref{TwierdzenieCoupling5}.
\end{proof}

\section*{Acknowledgments}
I would like to thank Andrzej Ruci\'{n}ski for the suggestion to read the article~\cite{Coupling}. 
I acknowledge a partial support by Ministry of Science and Higher Education, grant N N206 2701 33, 2007--2010.

\bibliographystyle{amsplain}
\normalfont
\bibliography{equivalence_f}

\providecommand{\bysame}{\leavevmode\hbox to3em{\hrulefill}\thinspace}
\providecommand{\MR}{\relax\ifhmode\unskip\space\fi MR }
% \MRhref is called by the amsart/book/proc definition of \MR.
\providecommand{\MRhref}[2]{%
  \href{http://www.ams.org/mathscinet-getitem?mr=#1}{#2}
}
\providecommand{\href}[2]{#2}
\begin{thebibliography}{10}

\bibitem{KsiazkaPoisson}
A.~D. Barbour, L.~Holst, and S.~Janson, \emph{Poisson approximation}, Oxford
  University Press, 1992.

\bibitem{RIGDegree2}
M.~Bloznelis, \emph{Degree distribution of a typical vertex in a general random
  intersection graph}, Lithuanian Mathematical Journal \textbf{48} (2008),
  no.~1, 38--45.

\bibitem{WSNphase2}
M.~Bloznelis, J.~Jaworski, and K.~Rybarczyk, \emph{Component evolution in a
  secure wireless sensor network}, Networks \textbf{53} (2009), no.~1, 19--26.

\bibitem{KsiazkaBollobas}
B.~Bollob\'{a}s, \emph{Random {G}raphs}, Academic Press, 1985.

\bibitem{GpEpidemics}
T.~Brittom, M.~Deijfen, A.~N. Lager\r{a}s, and M.~Lindholm, \emph{Epidemics on
  random graphs with tunable clustering}, Journal of Applied Probability
  \textbf{45} (2008), no.~3, 743--756.

\bibitem{RIGTunableDegree}
M.~Deijfen and W.~Kets, \emph{Random intersection graphs with tunable degree
  distribution and clustering}, Probab. Eng. Inform. Sci. \textbf{23} (2009),
  no.~4, 661--674.

\bibitem{GpEquivalence}
J.~A. Fill, E.~R. Scheinerman, and K.~B. Singer-Cohen, \emph{Random
  intersection graphs when $m=\omega(n)$: {A}n equivalence theorem relating the
  evolution of the ${G}(n, m, p)$ and ${G}(n, p)$ models}, Random Structures
  and Algorithms \textbf{16} (2000), 156--176.

\bibitem{RIGGodehardt1}
E.~Godehardt and J.~Jaworski, \emph{Two models of random intersection graphs
  for classification}, Studies in Classification, Data Analysis and Knowledge
  Organization (Opitz O. and Schwaiger M., eds.), vol.~22, Springer--Verlag,
  2003, pp.~67--81.

\bibitem{KsiazkaJLR}
S.~Janson, T.~{\L}uczak, and A.~Ruci\'{n}ski, \emph{Random {G}raphs}, Wiley,
  2001.

\bibitem{GpSubgraph}
M.~Karo\'{n}ski, E.~R. Scheinerman, and K.B. Singer-Cohen, \emph{On random
  intersection graphs: {T}he subgraph problem}, Combinatorics, Probability and
  Computing \textbf{8} (1999), 131--159.

\bibitem{Coupling}
J.~H. Kim, \emph{Perfect matchings in random uniform hypergraphs}, Random
  Structures and Algorithms \textbf{23} (2003), no.~2, 111 -- 223.

\bibitem{EREquivalence}
T.~{\L}uczak, \emph{On the equivalence of two basic models of random graphs},
  Random Graphs 87' (Karo\'{n}ski M., Jaworski J., and Ruci\'{n}ski A., eds.),
  John Wiley \& Sons, 1990, pp.~151--158.

\bibitem{WSNWlosi}
R.~Di Pietro, L.~V.Mancini, A.Mei, A.Panconesi, and J.Radhakrishnan,
  \emph{Sensor networks that are provably resilient}, Proc 2nd IEEE Int Conf
  Security Privacy Emerging Areas Commun Networks (SecureComm 2006), Baltimore,
  MD, 2006, 2006.

\bibitem{SingerPhD}
K.~B. Singer, \emph{Random intersection graphs}, Ph.D. thesis, Department of
  Mathematical Sciences, The Johns Hopkins University, 1995.

\bibitem{GpPoissonCliques}
D.~Stark and K.~Rybarczyk, \emph{Poisson approximation of the number of cliques
  in random intersection graphs}, Journal of Applied Probability \textbf{47}
  (2010), no.~3, 826--840.

\end{thebibliography}

\end{document}